\documentclass[a4paper]{article}

\usepackage[english]{babel}
\usepackage[utf8x]{inputenc}
\usepackage[T1]{fontenc}

\usepackage[a4paper,top=2cm,bottom=2.5cm,left=2.3cm,right=2.3cm,marginparwidth=1.75cm]{geometry}

\usepackage[colorinlistoftodos]{todonotes}
\usepackage[colorlinks=true, allcolors=blue]{hyperref}
\usepackage{xy}
\usepackage[affil-it]{authblk}
\usepackage{amssymb}
\usepackage{cancel}
\usepackage{epstopdf}
\usepackage{verbatim}
\usepackage{amsmath,amscd}
\usepackage{graphicx}
\usepackage{subcaption}
\usepackage{amsthm}
\usepackage{tikz}
\usetikzlibrary{cd}
\usepackage{pgf}
\usepackage{mathrsfs}
\usetikzlibrary{arrows}
\usepackage[toc,page]{appendix}
\usepackage{caption}
\usepackage[colorinlistoftodos]{todonotes}

\newtheorem{prop}{Proposition}

\makeatletter
\newtheorem*{rep@theorem}{\rep@title}
\newcommand{\newreptheorem}[2]{%
\newenvironment{rep#1}[1]{%
 \def\rep@title{#2 \ref{##1}}%
 \begin{rep@theorem}}%
 {\end{rep@theorem}}}
\makeatother

\newtheorem{theorem}{Theorem}[section]
\newtheorem{lemma}{Lemma}[section]
\newreptheorem{theorem}{Theorem}
\newtheorem{pro}{Proposition}[section]

\theoremstyle{definition}
\newtheorem{mydef}{Definition}[section]

\newtheorem{rem}{Remark}[section]
\newtheorem{exa}{Example}[section]
\newtheorem{cor}{Corollary}[section]

\title{Torsion points on elliptic curves and tame semistable coverings}
\author{Paul Alexander Helminck}
\affil{University of Bremen,\\
ALTA, Institute for Algebra, Geometry, Topology and their Applications}

\begin{document}
\maketitle

\definecolor{qqqqff}{rgb}{0,0,1}

\begin{abstract}
In this paper, we study tame Galois coverings of semistable models that arise from torsion points on elliptic curves. These coverings induce Galois morphisms of intersection graphs and we express the decomposition groups of the edges in terms of the reduction type of the elliptic curve. To that end, we first define the reduction type of an elliptic curve $E/K(C)$ on a subgraph of the intersection graph $\Sigma(\mathcal{C})$ of a strongly semistable model $\mathcal{C}$. In particular, we define the notions of good and multiplicative reduction on subgraphs of the intersection graph of $\Sigma(\mathcal{C})$. After this, we show that if an elliptic curve has good reduction on an edge and $\mathrm{char}(k)\nmid{N}$, then the $N$-torsion extension is unramified above that edge, as in the codimension one case. Furthermore, we prove a combinatorial version of a theorem by Serre on transvections for elliptic curves with a non-integral $j$-invariant. Our version states that the Galois representation $\rho_{\ell}:G\rightarrow{\mathrm{SL}_{2}(\mathbb{F}_{\ell})}$ of an elliptic curve with multiplicative reduction at an edge $e$ contains a transvection $\sigma\in{I_{e'/e}}$ for prime numbers $\ell$ that do not divide the slope of the normalized Laplacian of the $j$-invariant. 


\end{abstract}

\section{Introduction}

Let $K$ be a discretely valued field with valuation $v:K^{*}\rightarrow{\mathbb{Z}}$, uniformizer $\pi$, maximal ideal $\mathfrak{m}$, valuation ring $R$ and residue field $R/\mathfrak{m}=k$. We will assume that $K$ is complete, has characteristic zero and that $k$ is algebraically closed. Let $C$ be a smooth, projective, geometrically irreducible curve over $K$. We will assume that it is semistable over $K$ and we let $\mathcal{C}$ be a \emph{strongly semistable model} over $R$. That is, $\mathcal{C}$ is semistable and the irreducible components in the special fiber are smooth. For every \emph{tame} geometric Galois cover\footnote{In the sense that $C'$ is geometrically irreducible, $K(C)\rightarrow{K(C')}$ is Galois and the order of the Galois group is coprime to characteristic of the residue field $k$.} 
\begin{equation}
C'\rightarrow{C}
\end{equation}
with Galois group $G$,
there is a strongly semistable model $\mathcal{C}$ and an associated Galois cover $\mathcal{C}'\rightarrow{\mathcal{C}}$ of strongly semistable models over a finite extension $K'$ of $K$, see \cite[Chapter 10, Proposition 4.30]{liu2}. This in turn induces a Galois morphism of metrized complexes of $k$-curves
\begin{equation}
\Sigma(\mathcal{C}')\rightarrow{\Sigma(\mathcal{C})}
\end{equation} 
with Galois group $G$, see \cite[Proposition 4.7.2]{tropabelian}. We will study this Galois covering of metrized complexes of $k$-curves for the special case of Galois coverings arising from torsion points on elliptic curves.

Let $E/K(C)$ be an elliptic curve over $K(C)$ with identity element $\mathcal{O}$. 
We fix an $N\in\mathbb{N}$ and consider the finite extension of $K(C)$ induced by the $N$-torsion:
\begin{equation}
K(C)\subset{K(C)(E[N])}.
\end{equation}
This extension is Galois and we denote its Galois group by $G$. We will assume that $K$ contains a primitive $N$-th root of unity $\mu_{N}$.  After another suitable finite extension of the base field $K$, this extends to a semistable covering
\begin{equation}
\mathcal{C}_{N}\rightarrow{\mathcal{C}},
\end{equation}
giving rise to a corresponding morphism of metrized complexes $\Sigma(\mathcal{C}_{N})\rightarrow{\Sigma(\mathcal{C})}$.  
Choosing a basis $\{P_{1},P_{2}\}$ for $E[N]$, we then find a natural injection 
\begin{equation}
\rho_{}:G\rightarrow{\text{SL}_{2}(\mathbb{Z}/N\mathbb{Z})}.
\end{equation}
For the rest of the paper, we will be interested in this representation and the corresponding inertia and decomposition groups for edges and vertices. 

The intersection graph $\Sigma(\mathcal{C}_{N})$ consists of codimension one points (the vertices/irreducible components) and codimension two points (the edges/intersection points). For the codimension one points, we can use the reduction theory in \cite[Chapter VII]{silv1} and \cite[Chapter 7]{Bosch1990}. For the codimension two points, no such reduction theory is available. For instance, there is no general notion of a "N\'{e}ron model" over more general base schemes than Dedekind schemes, see \cite{Holmes1} for the case where $E\rightarrow{S}$ is \emph{semistable}.


To say something about the ramification at these codimension two points, we will use a trick to switch from codimension two points to codimension one points. The idea is to subdivide an edge to obtain a regularization and then to consider the inertia groups of the new components. The inertia groups of these codimension one points are then related to the inertia group of the original edge by Theorem \ref{SwitchTheorem}.

Before we consider the extensions coming from torsion points on elliptic curves, we first define the reduction type of an elliptic curve on subgraphs $T$ of the intersection graph $\Sigma(\mathcal{C})$ of a strongly semistable model $\mathcal{C}$. This will use the notion of normalized Laplacians, which will be defined and studied in Sections \ref{LaplacianSection} and \ref{NormalizedLaplacianSection}. The reduction type of an elliptic curve on subgraphs of $\Sigma(\mathcal{C})$ will then be defined in Section \ref{ReductionType2}. It will be shown that the notions of good and multiplicative reduction are stable under finite extensions $K\subset{K'}$. Furthermore, we will show that these notions of reduction types on subgraphs reduce to the normal notions for vertices in subdivisions of $\Sigma(\mathcal{C})$. After that, we will illustrate that having good reduction is a closed condition, in the sense that if an elliptic curve has good reduction on a graph $T$, then it also has good reduction on the corresponding complete graph $\overline{T}$. We will also show that these notions of good and multiplicative reduction are stable under disjointly branched morphisms.     


After these introductory sections, we will consider torsion extensions $K(C)\subset{K(C)(E[N])}$ coming from elliptic curves over $K(C)$. We will prove two theorems for inertia groups of edges in $\Sigma(\mathcal{C}_{N})$ that are 
probably  familiar from the theory of elliptic curves for the Dedekind case. The first theorem uses the N\'{e}ron-Ogg-Shafarevich criterion and the second theorem uses the Tate uniformization for elliptic curves with split multiplicative reduction. 
We write $\delta_{e}(\phi_{j})$ for the absolute value of the slope of the normalized Laplacian (see Definition \ref{NormalizedLaplacian}) corresponding to the $j$-invariant. We then have the following two theorems. 
\begin{reptheorem}{MainTheoremGoodReduction}
{\bf{(Good reduction for edges)}}
Let $E/K(C)$ be an elliptic curve over a curve $C$ with strongly semistable model $\mathcal{C}$ and intersection graph $\Sigma(\mathcal{C})$. Let $N$ be an integer with $\mathrm{char}(k)\nmid{N}$ and suppose that $E$ has good reduction on an edge $e$ of $\Sigma(\mathcal{C})$. Then $E[N]$ is unramified over $e$. 
\end{reptheorem}
\begin{reptheorem}{MainTheorem2}
{\bf{(Transvections for edges)}}
Let $E$ be an elliptic curve over $K(C)$ with multiplicative reduction at an edge $e\in\Sigma(\mathcal{C})$ and let $\mathcal{C}_{\ell}\rightarrow{\mathcal{C}}$ be the semistable covering associated to the Galois extension $K(C)\subset{K(C)(E[N])}$ with Galois group $G$ and representation $\rho_{\ell}:G\rightarrow{\mathrm{SL}_{2}(\mathbb{F}_{\ell})}$. Let $p=\text{char}(k)$  and let $\ell\geq{3}$ be a prime with $p\nmid{\#\rho_{\ell}(G)}$ and $\ell\nmid{\delta_{e}(\phi_{j})}$. Let $e'$ be an edge of $\Sigma(\mathcal{C}_{\ell})$ lying above $e$. There then exists a $\sigma\in{I_{e'/e}}$ such that
\begin{equation}
\rho_{\ell}(\sigma)\equiv{
  \begin{pmatrix}
1 & 1 \\
0&1 \end{pmatrix} 
}\mod{\ell}.
\end{equation}
\end{reptheorem}

For $G$ acting irreducibly on $E[\ell]$ (with $E$ as in Theorem \ref{MainTheorem2}), we then have that $G\rightarrow{\mathrm{SL}_{2}(\mathbb{F}_{\ell})}$ is surjective, as in \cite[IV-20, Lemma 2]{Ser2}. This is contained in Theorem \ref{SurjectivityRepresentation}. Checking the irreducibility of $G$ on $E[\ell]$ is then a matter of checking whether $E$ admits any rational $p$-isogenies to other elliptic curves. 


\section{Preliminaries}

In this section, we will give the notation we will use throughout this paper and a summary of the notions of disjointly branched morphisms, metrized complexes of $k$-curves and Laplacians. The reader is referred to \cite{tropabelian}, \cite{ABBR1} and \cite{supertrop} for more background on these topics. 

The notion of disjointly branched morphisms is introduced as a substitute for semistable covers to highlight the fact that the branch locus plays an important role. For these covers $\mathcal{C}\rightarrow{\mathcal{D}}$ (which are Galois by assumption throughout the paper), we obtain a Galois morphism of intersection graphs $\Sigma(\mathcal{C})\rightarrow{\Sigma(\mathcal{D})}$, in the sense that $\Sigma(\mathcal{C})/G=\Sigma(\mathcal{D})$. Moreover, assigning lengths and weights to the edges and vertices of $\Sigma(\mathcal{C})$ and $\Sigma(\mathcal{D})$, we obtain a Galois quotient of metrized complexes of $k$-curves. That is to say, we also have a quotient of the corresponding components and the lengths of the edges in $\Sigma(\mathcal{C})$ are multiplied by their ramification indices.

Given a strongly semistable model $\mathcal{C}$ of a curve $C$, one can associate to any principal divisor $\text{div}_{\eta}(f)$ (where $f\in{K(C)}$) its \emph{Laplacian} $\phi_{f}$, which is a piecewise linear function on the intersection graph $\Sigma(\mathcal{C})$ of $\mathcal{C}$. This Laplacian keeps track of the various vertical divisors associated to $f$ and one can use this to relate the divisor of $f$ to the divisors of the reductions $\overline{f^{\Gamma}}$ for irreducible components $\Gamma$. This relation is more commonly known as the Poincar\'{e}-Lelong or Slope formula and can be found in \cite[Theorem 3.3.2]{tropabelian} and \cite[Theorem 5.15]{BPRa1}. We will also use a version which is more directly related to the vertical divisors on $\mathcal{C}$, see Theorem \ref{MainThmVert}. After this, we will define the notion of a \emph{normalized} Laplacian and we will study its behavior under finite extensions. This will lead to the definition of the \emph{extended} normalized Laplacian on the corresponding metrized graph of $\Sigma(\mathcal{C})$.  

\subsection{Notation}

We will use the following notation throughout this paper:

\begin{itemize}
\item $K$ is a discretely valued complete field of characteristic zero with valuation $v:K^{*}\rightarrow{\mathbb{Z}}$,
\item $R=\{x\in{K}:v(x)\geq{0}\}$ is the valuation ring of $K$,
\item $R^{*}=\{x\in{K}:v(x)=0\}$ is the unit group,
\item $\mathfrak{m}=\{x\in{K}:v(x)>0\}$ is the unique maximal ideal in $R$,
\item $\pi$ is a uniformizer for $v$, i.e. $\pi{R}=\mathfrak{m}$,
\item $k:=R/\mathfrak{m}$ is the residue field of $R$.
\end{itemize}
We will assume that $v$ is normalized so that $v(\pi)=1$. For simplicity, we will also assume that the residue field $k$ is algebraically closed. In practice, it will be sufficient to assume that the residue field is large enough to contain the coordinates of all the branch and ramification points. 
For any finite extension $K'$ of $K$, we let $R'$ be a discrete valuation ring in $K'$ dominating $R$. For any Noetherian local ring $A$ with maximal ideal $\mathfrak{m}_{A}$, we let $\hat{A}$ be its $\mathfrak{m}_{A}$-adic completion, as in \cite[Section 1.3]{liu2}. Curves over a field in this paper will be smooth, geometrically irreducible and projective unless stated otherwise. The function field of any such curve will be denoted by $K(C)$ and its algebraic closure by $\overline{K(C)}$. Its absolute Galois group will be denoted by $G_{K(C)}$ and for any subextension $K(C)\subset{L}\subset{\overline{K(C)}}$ we denote the corresponding subgroup by $G_{L}$. In particular, if $L$ is finite Galois over $K(C)$, we find that $G_{L}$ is a normal subgroup of finite index. The Galois group of $L/K(C)$ is then $G_{K(C)}/G_{L}$.

\subsection{Metrized complexes of $k$-curves}

Let $\mathcal{C}$ be a strongly semistable model over $R$ of a curve $C$ over $K$. That is, it is semistable and the irreducible components in the special fiber are smooth. For the generic point $\mathfrak{p}\in\mathcal{C}_{s}$ of any irreducible component $\Gamma$, we assume that the valuation $v_{\mathfrak{p}}(\cdot{})$ is normalized such that it coincides with the discrete valuation on $R$. In other words, $v_{\mathfrak{p}}(\pi)=1$ (since $\mathcal{C}_{s}$ is reduced). In this section, we will associate a \emph{metrized complex of $k$-curves} to $\mathcal{C}$, which consists of a graph and a compatible assignment of a curve over $k$ to every vertex. We first define the intersection graph of $\mathcal{C}$. 

\begin{mydef}{\bf{(Dual Intersection Graph)}}
Let $\mathcal{C}$ be a strongly semistable model for a curve $C$ over $K$. Let $\{\Gamma_{1},...,\Gamma_{r}\}$ be the set of irreducible components. We define the dual intersection graph $\Sigma(\mathcal{C})$ of $\mathcal{C}_{}$ to be the finite graph whose vertices $v_{i}$ correspond to the irreducible components $\Gamma_{i}$ of $\mathcal{C}_{s}$ and whose edges correspond to intersections between components. The latter means that we have one edge for every point of intersection. 
We write $V(\Sigma(\mathcal{C}))$ for the vertex set and $E(\Sigma(\mathcal{C}))$ for the edge set of $\Sigma(\mathcal{C})$. 
\end{mydef}

We will now assign some additional data to this intersection graph to turn it into a weighted metric graph. We will define this notion first.
\begin{mydef}
A (finite) {\bf{weighted metric graph}} consists of a triple $(\Sigma, l(\cdot{}),w(\cdot{}))$, where $\Sigma$ is a finite graph, $l(\cdot{})$ a length function
\begin{equation}
l:E(\Sigma)\rightarrow{\mathbb{N}}
\end{equation}
and $w(\cdot{})$ a weight function
\begin{equation}
w: V(\Sigma)\rightarrow{\mathbb{N}}.
\end{equation}
\end{mydef}

For any strongly semistable model $\mathcal{C}$ of $C$, we now assign a length function and a weight function to $\Sigma(\mathcal{C})$, turning it into a weighted metric graph. We use the following fact: 
for every intersection point $z\in\mathcal{C}$, we have that
\begin{equation}
 \hat{\mathcal{O}}_{\mathcal{C},z}\simeq{R[[x,y]]/(xy-\pi^{n})},
\end{equation}
see \cite[Chapter 10, Corollary 3.22]{liu2}. 
We then define $l:E(\Sigma(\mathcal{C}))\rightarrow{\mathbb{N}}$ by
\begin{equation}
l(e)=n.
\end{equation}
The weight function 
$w:V(\Sigma)\rightarrow{\mathbb{N}}
$ is then defined by
\begin{equation}
w(v_{i}):=g(\Gamma_{i}),
\end{equation}
where $g(\Gamma_{i})$ is the arithmetic genus of $\Gamma_{i}$. We will refer to the triple $\{(\Sigma(\mathcal{C}),\ell(\cdot{}),w(\cdot{}))\}$ as the weighted metric graph associated to $\mathcal{C}$. 

To every vertex $v$ in $\Sigma(\mathcal{C})$, we now add a curve $C_{v}/k$ and an identification of the edges adjacent to $v$ with $k$-rational points of $C_{v}$. This will turn our weighted metric graph into a metrized complex of $k$-curves, see \cite[Definition 4.7.1]{tropabelian} and \cite[Definition 2.17]{ABBR1} for more on this concept.   

\begin{mydef}
A {\bf{metrized complex}} of $k$-curves consists of the following data:
\begin{enumerate}
\item A weighted metric graph $(\Sigma, w(\cdot{}), l(\cdot{}))$,
\item A smooth, irreducible projective curve $C_{v}/k$ for every vertex $v\in{V(\Sigma)}$ such that $w(v)=g(C_{v})$.
\item An injective function $\text{red}_{v}:T_{v}\rightarrow{C_{v}(k)}$, where $T_{v}$ is the set of edges connected to $v$.
\end{enumerate}
The metrized complex $k$-curves corresponding to this data will be denoted by $(\Sigma,w(\cdot{}),l(\cdot{}),\text{red}_{v}(\cdot{}))$, or just $\Sigma$ if no confusion can arise.
\end{mydef}

Now let $\mathcal{C}$ be a strongly semistable model for $C$. Let $(\Sigma(\mathcal{C}), w(\cdot{}), l(\cdot{}))$ be the weighted metric graph associated to $\mathcal{C}$. A vertex $v$ naturally comes from a smooth irreducible projective curve $\Gamma$, so we assign this curve to the vertex $v$. Furthermore, we have that every edge $e$ corresponds to an intersection point $z$ of two components $\Gamma$ and $\Gamma'$. We assign the induced point $\tilde{z}\in\Gamma$ to $e$.  We thus see that we have a natural metrized complex of $k$-curves associated to $\mathcal{C}$, which we will again denote by $\Sigma(\mathcal{C})$. 



\subsection{Disjointly branched morphisms}\label{DisBran3}

Let $\phi:C\rightarrow{D}$ be a finite, Galois morphism of smooth, projective, geometrically irreducible curves over $K$ and let $G$ be its Galois group. We will assume that the order of $G$ is coprime to the characteristic of $k$. In other words, $\phi$ is \emph{tame}.  Let $\mathcal{D}$ be any strongly semistable model of $D$ over $R$. Let $\mathcal{C}$ be the normalization of $\mathcal{D}$ in the function field $K(C)$ of $C$ and consider the induced finite Galois morphism $\phi_{\mathcal{C}}:\mathcal{C}\rightarrow{\mathcal{D}}$.
\begin{mydef}
We say that $\phi_{\mathcal{C}}$ is \emph{disjointly branched} if the following hold:
\begin{enumerate}
\item The closure of the branch locus in $\mathcal{D}$ consists of disjoint, smooth sections over $\text{Spec}(R)$.
\item Let $y$ be a generic point of an irreducible component in the special fiber of $\mathcal{C}$. Then the induced morphism $\mathcal{O}_{\mathcal{D},\phi_{\mathcal{C}}(y)}\rightarrow{\mathcal{O}_{\mathcal{C},y}}$ is \'{e}tale.
\end{enumerate}
\end{mydef}

By \cite[Proposition 4.1.1.]{tropabelian} (which uses \cite[Theorem 2.3.]{liu1}), we then find that $\mathcal{C}$ is also strongly semistable, so we obtain a finite morphism of strongly semistable models $\mathcal{C}\rightarrow{\mathcal{D}}$ over $R$. By a series of blow-ups, we can always find a strongly semistable model $\mathcal{D}$ such that the branch locus is disjoint. After removing the ramification on the vertical components by a finite extension $K\subset{K'}$, we then obtain a disjointly branched morphism $\phi_{\mathcal{C}}:\mathcal{C}\rightarrow{\mathcal{D}}$ by taking the normalization of $\mathcal{D}$ in $K(C)$. We will refer to the morphism $\phi_{\mathcal{C}}$ created in this way as a semistable covering (or disjointly branched morphism) for $C\rightarrow{D}$. 

We now have the following
\begin{theorem}
The finite Galois morphism $\mathcal{C}\rightarrow{\mathcal{D}}$ induces a quotient
\begin{equation}
\Sigma(\mathcal{C})/G=\Sigma(\mathcal{D})
\end{equation}
of metrized complexes of $k$-curves.
\end{theorem}
\begin{proof}
This is \cite[Theorem 4.6.1, Proposition 4.7.1 and Proposition 4.7.2]{tropabelian}. We refer the reader the reader to that thesis for the details. We also invite the reader to compare this to the results in  \cite[Section 7]{ABBR1}.
\end{proof}

\subsection{Laplacians}\label{LaplacianSection}

Let $\Sigma$ be a graph, which we will assume to be finite, connected and without loop edges. Let $V(\Sigma)$ be its vertices and $E(\Sigma)$ its edges. We define $\text{Div}(\Sigma)$ to be the free abelian group on the vertices of $\Sigma$. 
Writing $D\in{\text{Div}(\Sigma)}$ as $D=\sum_{v\in{V(\Sigma)}}c_{v}(v)$, we define the degree map as $\text{deg}(D)=\sum_{v\in{V(\Sigma)}}c_{v}$. We let $\text{Div}^{0}(\Sigma)$ be the group of divisors of degree zero on $\Sigma$.

 Now let $\mathcal{M}(\Sigma)$ be the group of $\mathbb{Z}$-valued functions on $V(\Sigma)$. Define the {\bf{Laplacian operator}} $\Delta:\mathcal{M}(\Sigma)\longrightarrow\text{Div}^{0}(\Sigma)$ by 
\begin{equation}\label{LaplacianDefinition}
\Delta(\phi)=\sum_{v\in{V(\Sigma)}}\sum_{e=vw\in{E(\Sigma)}}(\phi(v)-\phi(w))(v).
\end{equation}   
The fact that the image of $\Delta$ is in $\text{Div}^{0}(\Sigma)$ is contained in \cite[Corollary 1]{bakerfaber}, where it is found as a consequence of the self-adjointness of the Laplacian operator. 

We then define the group of principal divisors to be the image of the Laplacian operator:
\begin{equation*}
\text{Prin}(\Sigma):=\Delta(\mathcal{M}(\Sigma)).
\end{equation*}

Let $\mathcal{C}$ be a strongly semistable \emph{regular} model of a curve $C/K$ and let $\text{Prin}(C)$ be the group of principal divisors on $C$. We now give a specialization map from the group $\text{Div}(\mathcal{C})$ of Cartier divisors on $\mathcal{C}$ to the divisor group $\text{Div}(\Sigma(\mathcal{C}))$. 
This will use the intersection pairing on $\mathcal{C}$, see \cite[Section 2.2]{tropabelian}. 
Let $\rho:\mathrm{Div}(\mathcal{C})\rightarrow{\mathrm{Div}(\Sigma(\mathcal{C}))}$ be the homomorphism given by
\begin{equation}
\rho(\mathcal{D})=\sum_{v_{i}\in\Sigma(\mathcal{C})}(\mathcal{D}\cdot{\Gamma_{i}})(v_{i}).
\end{equation}
We will call this the {\emph{specialization}} homomorphism. Note that it is indeed a homomorphism by bilinearity of the intersection map. For linearly equivalent divisors $D$ and $D'$, we have $D\cdot{E}=D'\cdot{E}$, see \cite[Chapter 9, Theorem 1.12.b]{liu2}. This then shows that $\text{Prin}(\mathcal{C})$ is in the kernel of $\rho$. A simple calculation on the vertical divisors (see \cite[Lemma 3.3.1]{tropabelian}) then shows that vertical divisors are mapped to principal divisors on $\Sigma(\mathcal{C})$. This then gives 
\begin{lemma}\label{Principaldivisors}
The specialization map $\rho$ induces a map
\begin{equation}
\mathrm{Prin}({C})\rightarrow{\mathrm{Prin}(\Sigma(\mathcal{C}))}.
\end{equation}
\end{lemma}

By Lemma \ref{Principaldivisors} we find that for any $f\in{K(C)}$, there exists a $\phi\in{\mathcal{M}(\Sigma(\mathcal{C}))}$ such that $\rho((f))=\Delta(\phi)$. We refer to this $\phi$ as the \emph{Laplacian} of $f$. This $\phi$ can tell us a great deal about the various \emph{vertical} divisors corresponding to $f$, as we will see in Theorem \ref{MainThmVert}. 

Note that by taking the closure of a principal divisor in ${C}$, one obtains a Cartier divisor which might \emph{not} be principal. Indeed, the same generator might vanish on some irreducible component in the special fiber $\mathcal{C}_{s}$. For any component $\Gamma$, we then consider the $\Gamma$-modified form of any $f\in{K(C)}$, which is defined by
\begin{equation}
f^{\Gamma}=\dfrac{f}{\pi^{k}},
\end{equation} 
where $k:=v_{\Gamma}(f)$. By definition, we then have $v_{\Gamma}(f^{\Gamma})=0$. We can then express the \emph{vertical} divisor corresponding to $f^{\Gamma}$ in terms of the Laplacian function $\phi$ by the following
\begin{theorem}\label{MainThmVert}
Let $\rho(\text{div}_{\eta}(f))$ be the induced principal divisor of $f$ on the intersection graph $\Sigma(\mathcal{C})$ of $\mathcal{C}$. 
Write
\begin{equation*}
\Delta(\phi)=\rho(\text{div}(f))
\end{equation*}
for some $\phi:\mathbb{Z}^{V}\longrightarrow{\mathbb{Z}}$. 
Choose $\phi$ such that $\phi(\Gamma)=0$. Then the unique vertical divisor corresponding to $\text{div}_{\eta}(f)$ with $V_{f^{\Gamma}}(\Gamma)=0$ is given by
\begin{equation}\label{ExplVert2}
V_{f^{\Gamma}}=\sum_{i}\phi(\Gamma_{i})\cdot{\Gamma_{i}}.
\end{equation} 
\end{theorem}
\begin{proof}
See \cite[Theorem 3.3.1]{tropabelian}. 
\end{proof}

This then also quickly gives the \emph{Poincar\'{e}-Lelong formula}:

\begin{theorem}
{\bf[Poincar\'{e}-Lelong formula]}
Let $\mathcal{C}$ be a strongly semistable regular model of a curve $C$ with $f\in{K(\mathcal{C})}$.
 Let $\tilde{x}$ be an intersection point of an irreducible component $\Gamma$ with another irreducible component $\Gamma'$. Then
\begin{equation}
v_{\tilde{x}}(\overline{f^{\Gamma}})=\phi(v')-\phi(v),
\end{equation}
where $\phi$ is the Laplacian associated to $f$. 
\end{theorem}
\begin{proof}
See \cite[Theorem 3.3.2]{tropabelian} and \cite[Theorem 5.15, part 5]{BPRa1}. 
\end{proof}

\begin{lemma}\label{NormalizedLaplacianLemma}
Let $f\in{K(C)}$ and let $\mathcal{C}$ be a strongly semistable regular model for $C$. There is then a unique $\phi\in\mathcal{M}(\Sigma(\mathcal{C}))$ such that $v_{\mathfrak{p}}(f)=\phi(w_{\mathfrak{p}})$ for every $w_{\mathfrak{p}}\in{V(\Sigma(\mathcal{C}))}$, where $\mathfrak{p}$ is the generic point of the irreducible component corresponding to $w_{\mathfrak{p}}$.
\end{lemma}
\begin{proof}
There can only be at most one such function $\phi:V(\Sigma(\mathcal{C}))\rightarrow{\mathbb{Z}}$, so the unicity is clear. Let $\mathfrak{p}$ be the generic point of any irreducible component $\Gamma$ in the special fiber. If $v_{\mathfrak{p}}(f)=0$, then the Lemma follows from Theorem \ref{MainThmVert}. For the general case, let $k$ be such that $v_{\mathfrak{p}}(f)=k$ and let $f^{\Gamma}=\dfrac{f}{\pi^{k}}$. We let $\phi$ be any function such that $v_{\mathfrak{q}}(f^\Gamma)=\phi(w_{\mathfrak{q}})$ for $\mathfrak{q}$ a generic point of any irreducible component in $\mathcal{C}_{s}$. The function $\phi'=\phi(\cdot{})+k$ then solves the problem in the Lemma. 
\end{proof}

%
\begin{mydef}\label{NormalizedLaplacian}
Let $f\in{K(C)}$ and let $\mathcal{C}$ be a strongly semistable regular model for $C$. 
The {\bf{normalized Laplacian}} corresponding to $f$ is then the unique function $\phi_{f}$ such that $v_{\mathfrak{p}}(f)=\phi_{f}(w_{\mathfrak{p}})$ for every $w_{\mathfrak{p}}\in{V(\Sigma(\mathcal{C}))}$. 
\end{mydef}

\subsection{Normalized Laplacians under finite extensions and inequalities}\label{NormalizedLaplacianSection}

We would now like to study the normalized Laplacians introduced in Definition \ref{NormalizedLaplacian} under finite extensions $K\subset{K'}$. Let $R'\subset{K'}$ be a discrete valuation ring dominating $R$ and let us suppose that the extension $R\subset{R'}$ is ramified of degree $n>{1}$. We let $\pi_{K'}\in{R'}$ be a uniformizer and first set $v_{K'}(\pi_{K'})=1$. Note that this valuation does \emph{not} extend $v_{K}$. The model $\mathcal{C}\times_{R}{R'}$ is not regular anymore by assumption on $n$, so we take a desingularization (see \cite[Chapter 9, Example 3.53]{liu2}) and obtain a morphism $\mathcal{C}'\rightarrow{\mathcal{C}\times_{R}{R'}}$, where $\mathcal{C}'$ is regular. The intersection graph of $\mathcal{C}'$ is obtained from that of $\mathcal{C}$ by subdividing every edge into $n$ edges and $n-1$ components.

Let $f\in{K(C)}\subset{K(C'):=K(C)\otimes_{K}{K'}}$. We have a natural inclusion of vertex sets $V(\Sigma(\mathcal{C}))\rightarrow{V(\Sigma(\mathcal{C}'))}$ since the morphism $\mathcal{C}'\rightarrow{\mathcal{C}}$ is an isomorphism outside the set of intersection points (corresponding to the edges). For any function $\phi\in{\mathcal{M}(\Sigma(\mathcal{C}'))}$, we thus obtain by restriction an element $\phi|_{\Sigma(\mathcal{C})}\in\mathcal{M}(\Sigma(\mathcal{C}))$.
\begin{lemma}\label{ScalingLemma}
Let $f\in{K(C)}\subset{K(C')}$ and consider the normalized Laplacians $\phi'_{f}\in\mathcal{M}(\Sigma(\mathcal{C}'))$ and $\phi_{f}\in{\mathcal{M}(\Sigma(\mathcal{C}))}$. We then have
\begin{equation}
\phi'_{f}|_{\Sigma(\mathcal{C})}=n\cdot\phi_{f},
\end{equation}
where $n$ is the ramification degree of $R\subset{R'}$. Furthermore, the values of $\phi'_{f}$ on other vertices are obtained by linearly interpolating the values on $\Sigma(\mathcal{C})$. That is, for any edge $e=vw\in{\Sigma(\mathcal{C})}$ subdivided into new vertices $\{v_{1},v_{2},...,v_{n-1}\}$ with generic points $\mathfrak{p}_{i}\in\mathcal{C}'_{s}$, we have
\begin{equation}
v_{\mathfrak{p}_{i}}(f)=(\phi_{f}(v')-\phi_{f}(v))\cdot{i}+n\cdot{\phi_{f}(v)}.
\end{equation}
\end{lemma}
\begin{proof}
The first claim follows from our assumption $v_{K'}(\pi_{K'})=1$. For the second claim, let $\rho'$ be the specialization map on $\mathcal{C}'$. The support of $\rho'(\mathrm{div}_{\eta}(f))=\rho(\mathrm{div}_{\eta}(f))$ lies in the image of $\Sigma(\mathcal{C})$ in $\Sigma(\mathcal{C}')$, so the sum of the "slopes" (as in Equation \ref{LaplacianDefinition}) of the normalized Laplacian cannot be nonzero at the new vertices $v\in\Sigma(\mathcal{C}')\backslash{\Sigma(\mathcal{C})}$. We thus see that the values of $\phi'_{f}$ on new vertices can only be obtained by linearly interpolating the values of $\phi_{f}$ on old vertices. The slope of this "line" is given by $\phi_{f}(v')-\phi_{f}(v)$ and the constant value for $i=0$ is $\phi_{f}(v)$, so we obtain the formula of the Lemma.  
\end{proof}

Lemma \ref{ScalingLemma} tells us that the normalized Laplacian will be invariant if we redefine our valuation $v_{K'}$. Let us set  $v_{K'}(\pi_{K'})=1/n$, where $n$ is the ramification degree of $R\subset{R'}$. The valuation of $v_{\mathfrak{p}}(f)$ is then scaled by $n$ as well, so we obtain the formula
\begin{equation}
v_{\mathfrak{p}_{i}}(f)=(\phi_{f}(v')-\phi_{f}(v))\cdot{i/n}+\phi_{f}(v).
\end{equation}

That is, if we consider the edge $e$ as an interval $[0,1]$, then the value of the normalized Laplacian at the new components corresponding to the $\mathfrak{p}_{i}$ is given by evaluating the function
\begin{equation}\label{ExtendedEquations}
\overline{\phi}_{f,e}(x)=(\phi_{f}(v')-\phi_{f}(v))\cdot{x}+\phi_{f}(v)
\end{equation}
at the points $i/n$ for $i\in\{0,1,...,n\}$. Note that this formula for $\overline{\phi}_{f,e}$ can be used for any $n$ (and thus for any finite extension). By viewing every edge $e=vw$ as an interval $[0,1]$ and identifying the endpoints for vertices, one obtains what is often known as the {\bf{metrized graph}} associated to $\mathcal{C}$.  See \cite[Section 1.4]{bakerfaber} for more on this subject. The metrized graph associated to $\mathcal{C}$ will be denoted by $\overline{\Sigma}(\mathcal{C})$. Using Equation \ref{ExtendedEquations}, one can then extend the normalized Laplacian $\phi_{f}$ to a continuous piecewise linear function $\overline{\phi}_{f}:\overline{\Sigma}(\mathcal{C})\rightarrow{\mathbb{R}}$. The function $\overline{\phi}_{f}$ coincides with $\phi_{f}$ at the original vertices $V(\Sigma(\mathcal{C}))$ and it gives the valuation of $f$ at new vertices obtained from desingularizing over finite extensions by the considerations in and after Lemma \ref{ScalingLemma}. %
\begin{mydef}
We call the continuous piecewise linear function $\overline{\phi}_{f}:\overline{\Sigma}(\mathcal{C})\rightarrow{\mathbb{R}}$ defined above the {\bf{extended normalized Laplacian}} of $f$. 
\end{mydef}



We now define \emph{inequalities} for normalized Laplacians. These inequalities will be used in defining reduction types for elliptic curves, see Definition \ref{ReductionType}. For the remainder of this section, we will assume that valuations on extensions of the base field $K$ extend the valuation of $K$. 

\begin{mydef}\label{InequalitiesDefinition}{\bf{(Inequalities for normalized Laplacians)}}
For any $f\in{K(C)}$, let $\phi_{f,T}$ be its normalized Laplacian and let $\overline{\phi}_{f}:\overline{\Sigma}(\mathcal{C})\rightarrow{\mathbb{R}}$ be the extended normalized Laplacian on the metrized graph associated to a strongly semistable regular model $\mathcal{C}$. Let $T$ be any subgraph of $\Sigma(\mathcal{C})$ and let $\overline{T}$ be its image in $\overline{\Sigma}(\mathcal{C})$. 
We write $\phi_{f,T}>0$ if $\overline{\phi}_{f}(x)>0$ for every $x\in\overline{T}$. We similarly define $\phi_{f,T}=0$, $\phi_{f,T}\geq{0}$, $\phi_{f,T}<0$ and $\phi_{f,T}\leq{0}$. 
\end{mydef}

\subsection{Edges and codimension one phenomena}

In this section, we recall a result obtained in \cite{tropabelian}, saying that one can recover the decomposition group of an edge by subdividing the edge and then considering the ramification of the corresponding components. 

Consider a disjointly branched Galois morphism $\phi:\mathcal{C}\rightarrow{\mathcal{D}}$ with $x\in\mathcal{C}$ an intersection point with length $n_{x}$ and $y$ its image in $\mathcal{D}$ with length $n_{y}$. We will denote the Galois group by $G$.
Let $y$ be an intersection point in $\mathcal{D}$, with corresponding components $\Gamma_{0}$ and $\Gamma_{n}$. Here $n$ is the length of $y$.
We now take a regular subdivision $\mathcal{D}_{0}$ of $\mathcal{D}$ in $y$. That is, we have a model $\mathcal{D}_{0}$ with a morphism $\psi: \mathcal{D}_{0}\rightarrow{\mathcal{D}}$ that is an isomorphism outside $y$ and the pre-image $\psi^{-1}\{y\}$ of $y$ consists of $n-1$ projective lines $\Gamma_{i}$. Here, the projective lines are labeled such that $\Gamma_{i}$ intersects $\Gamma_{i+1}$ in one point: $y_{i,i+1}$. Furthermore, we have that $\Gamma_{1}$ intersects an isomorphic copy of the original component $\Gamma_{0}$ in $y_{0,1}$ and likewise $\Gamma_{n-1}$ intersects an isomorphic copy of the original component $\Gamma_{n}$ in $y_{n-1,n}$, see \cite[Chapter 8, Example 3.53. and Chapter 9, Lemma 3.21]{liu2} 
for the details. 

 We now take the normalization $\mathcal{C}_{0}$ of $\mathcal{D}_{0}$ in $K(\mathcal{C})$. By virtue of the universal property for normalizations, we have a natural morphism
\begin{equation}
\mathcal{C}_{0}\rightarrow{\mathcal{C}}
\end{equation}
that is an isomorphism outside $\phi^{-1}(y)$.

Taking the tamely ramified extension $K\subset{K'}$ of order $\text{lcm }(|I_{\Gamma_{i}}|)$, we obtain a new model $\mathcal{D}'_{0}=\mathcal{D}_{0}\times_{\text{Spec}(R)}{\text{Spec}(R')}$ over $R'$, which is the normalization of $\mathcal{D}_{0}$ in $K'(\mathcal{D})$.
 Taking the normalization $\mathcal{C}'_{0}$ of this model inside $K'(\mathcal{C})$, we then naturally obtain morphisms
\begin{equation}
\mathcal{C}'_{0}\rightarrow{\mathcal{C}_{0}}\rightarrow{\mathcal{C}}.
\end{equation}
Here the first morphism is finite and the second one is birational. Note that by \cite[Chapter 10, Proposition 4.30]{liu2}, we have that $\mathcal{C}'_{0}$ is again semistable and that $G$ naturally acts on $\mathcal{C}_{0}$ and  $\mathcal{C}'_{0}$ such that $\mathcal{C}_{0}/G=\mathcal{D}_{0}$ and $\mathcal{C}'_{0}/G=\mathcal{D}'_{0}$ (which follows from the fact that $G$ acts naturally on any normalization, see \cite[Proposition 4.2.4]{tropabelian}). 

We now wish to study the inertia groups of the various points in $\mathcal{C}'_{0}$, $\mathcal{C}_{0}$ and $\mathcal{C}$. To do this, we will introduce the notion of a "\emph{chain}".

\begin{mydef}
Let $y_{i,i+1}$ and $y'_{i,i+1}$ be the intersection points in $\mathcal{D}_{0}$ and $\mathcal{D}'_{0}$ respectively that map to $y\in\mathcal{D}$ under the natural morphism. Similarly, let $y_{i}$ and $y'_{i}$ be the generic points of the components in $\mathcal{D}_{0}$ and $\mathcal{D}'_{0}$ that map to $y$. Here the generic points are labeled such that $y_{i,i+1}$ is a specialization of both $y_{i}$ and $y_{i+1}$. 
A {\bf{chain}} lying above these points is a collection of generic points $x_{i}$ in the special fiber of $\mathcal{C}_{0}$ or $\mathcal{C}'$ and closed points $x_{i,i+1}$ in $\mathcal{C}_{0}$ or $\mathcal{C}'_{0}$ such that:
\begin{enumerate}
\item $x_{i,i+1}$ is a specialization of both $x_{i}$ and $x_{i+1}$,
\item The $x_{i,i+1}$ map to $y_{i,i+1}$,
\item The $x_{i}$ map to $y_{i}$.
\end{enumerate}
For the remainder of this section, we will refer to these simply as a "{\it{chain}}".
\end{mydef}  

\begin{lemma}\label{ChainLemma}
For every intersection point $x\in\mathcal{C}$, there is only one chain in $\mathcal{C}_{0}$ and in $\mathcal{C}'_{0}$ lying above it.
\end{lemma}

For any intersection point $x$, we thus have a natural chain in $\mathcal{C}_{0}$ and $\mathcal{C}'_{0}$ associated to it. We then have the following Theorem regarding this chain.

\begin{theorem}\label{SwitchTheorem}
Let $x_{i}$ be the generic points of the unique chain in $\mathcal{C}_{0}$ associated to $x$ and let $I_{x_{i}}$ be their inertia groups. Then 
\begin{equation}
|I_{x_{i}}|=\dfrac{|I_{x}|}{\gcd(i,|I_{e}|)}.
\end{equation}
For $i$ such that $\gcd(i,|I_{x}|)=1$, we have that 
\begin{equation}
I_{x_{i}}=I_{x}.
\end{equation}
\end{theorem}
\begin{proof}
See \cite[Theorem 5.5.1]{tropabelian} and \cite[Proposition 5.5.2]{tropabelian}.
\end{proof}

In other words, the inertia groups for the components in a minimal desingularization $\mathcal{D}_{0}$ give the inertia group for the edge $x$. Since the residue field $k$ is algebraically closed, we have that the inertia group of an edge is equal to the decomposition group, so this will give the splitting behavior of an edge in a finite Galois extension.

\section{Torsion points on elliptic curves and Galois representations}

In this section, we review some material on Galois extensions $K(C)\subset{K(C)(E[N])}$ arising from torsion points on elliptic curves $E/K(C)$. They can be seen as analogues of the field extensions arising for the modular curves \begin{equation}
X(N)\rightarrow{X_{1}(N)}\rightarrow{X_{0}(N)}\rightarrow{\mathbb{P}^{1}_{j}}.
\end{equation}
These extensions give rise to representations
\begin{equation}
\rho:G\rightarrow{\mathrm{SL}_{2}(\mathbb{Z}/N\mathbb{Z})},
\end{equation}
given by letting the Galois group $G$ act on a basis $\{P_{1},P_{2}\}$ of the $N$-torsion of $E$. 
To these Galois extensions $K(C)\subseteq{K(C)(E[N])}$, we will assign disjointly branched morphisms $\mathcal{C}_{N}\rightarrow{\mathcal{C}}$, as defined in Section \ref{DisBran3}. We can thus associate Galois morphisms of metrized complexes $\Sigma(\mathcal{C}_{N})\rightarrow{\Sigma(\mathcal{C})}$ to these torsion points. These morphisms will be the main focus in this paper. 


\subsection{Galois representations}\label{Torsionextensions}

Let $E$ be an elliptic curve over the function field $K(C)$ of a fixed curve $C$. We will denote its identity element by $\mathcal{O}$. Furthermore, we will fix an $N\in\mathbb{N}$ for this section and assume that $\mu_{N}\in{K}$. $E$ is naturally an abelian variety (see \cite[Chapter 10, Lemma 2.8.]{liu2} or \cite{Milne1}) over $K(C)$ and we can consider its group of $N$-torsion points over $\overline{K(C)}$:
\begin{equation}
E[N]=\{P\in{E(\overline{K(C)})}:N\cdot{}P=\mathcal{O}\}.
\end{equation}
As a group, we have that $E[N]$ is isomorphic to $(\mathbb{Z}/N\mathbb{Z})^{2}$. Furthermore, there is an action of the absolute Galois group $G_{K(C)}$ on $E[N]$, turning it into a $G_{K(C)}$-module. We choose a basis $\{P_{1},P_{2}\}$ for $E[N]$ and then obtain a homomorphism
\begin{equation}
\rho:G_{K(C)}\rightarrow{\text{GL}_{2}(\mathbb{Z}/N\mathbb{Z})}.
\end{equation}
By the assumption that $\mu_{N}\in{K}$, we find using the Weil pairing that the image of $\rho$ lies in $\text{SL}_{2}(\mathbb{Z}/N\mathbb{Z})$, see \cite[Chapter III, Section 8]{silv1} 
and \cite[Page 42]{ModForms1}. We will denote the fixed field of the kernel of $\rho$ by $K(C)(E[N])$. The extension $K(C)\subset{K(C)(E[N])}$ is then finite Galois and we denote its Galois group by $G:=G_{K(C)}/G_{K(C)(E[N])}$.   We thus have an induced injective representation
\begin{equation}
\rho:G\rightarrow{\text{SL}_{2}(\mathbb{Z}/N\mathbb{Z})}.
\end{equation}

\begin{exa}
We refer the reader to \cite[Pages 41-100]{ModForms1} for the example below. Consider the elliptic curve $E$ over $K(j)$ defined by 
\begin{equation}
y^2+xy=x^3-\dfrac{36}{j-1728}x-{\dfrac{1}{j-1728}}.
\end{equation}

The $j$-invariant of this elliptic curve is equal to $j$. Furthermore, we have that the discriminant is given by
\begin{equation*}
\Delta=\dfrac{j^2}{(j-1728)^3}.
\end{equation*}

We thus see that for any $j_{0}\in{K}$ not equal to $0$ or $1728$, we obtain an elliptic curve over $K$ with $j$-invariant equal to $j_{0}$. 


One can then consider its $N$-torsion for any $N\in{\mathbb{N}}$. We choose a basis $\{P_{1},P_{2}\}$ and obtain the field extension $K(j)\subseteq{K(j)(E[N])}$, with corresponding representation
\begin{equation}
\rho:\mathrm{Gal}(K(j)(E[N])/K(j))\rightarrow{\mathrm{SL}_{2}(\mathbb{Z}/N\mathbb{Z})}.
\end{equation} 
We then consider the normal subgroup 
\begin{equation}
H:=\{\sigma\in\mathrm{Gal}(K(j)(E[N])/K(j)):\sigma(P)=\pm{P}\text{ for every }P\in{E[N]}\}.
\end{equation}
 We denote the fixed field of $H$ by $K(j)(E[N]/\pm)$. We then obtain a natural representation
 \begin{equation}
 \rho:\mathrm{Gal}(K(j)(E[N])/K(j))\rightarrow{\mathrm{SL}_{2}(\mathbb{Z}/N\mathbb{Z})/\{\pm{I}\}}.
 \end{equation} 

This induces a finite map of smooth curves $X(N)\rightarrow{\mathbb{P}^{1}_{K}}$. We call this curve $X(N)$ the \emph{full modular curve} of level $N$, as in \cite[Proposition 2, Page 48]{ModForms1}. Similarly, one also obtains $X_{0}(N)$ and $X_{1}(N)$ by considering the subgroups $H_{0}$ and $H_{1}$ of $\text{SL}_{2}(\mathbb{Z}/N\mathbb{Z})$ given by
\begin{align*}
H_{0}&=\left\{
\begin{pmatrix}
a & 0\\
b & d\\
\end{pmatrix}\in\text{SL}_{2}(\mathbb{Z}/N\mathbb{Z})
\right\},\\
H_{1}&=\left\{
\begin{pmatrix}
a & 0\\
b & \pm{}1\\
\end{pmatrix}\in\text{SL}_{2}(\mathbb{Z}/N\mathbb{Z})
\right\}.
\end{align*}
 of $\text{SL}_{2}(\mathbb{Z}/N\mathbb{Z})$. Note that $\pm{I}\in{H_{0},H_{1}}$ and $H_{1}\subset{H_{0}}$, so we obtain induced maps $X(N)\rightarrow{X_{1}(N)}\rightarrow{X_{0}(N)}\rightarrow{\mathbb{P}^{1}_{j}}$. The torsion extensions we consider in this paper can be seen as generalizations of these modular curves.  
\end{exa}

\begin{exa}
Let $E$ be the elliptic curve given by
\begin{equation}
y^2=x^3+Ax+B,
\end{equation}
where $A,B\in{K(C)}$. Let $N=2$. Then the field $K(C)(E[2])$ is none other than the splitting field of $f:=x^3+Ax+B$. The group $\text{SL}_{2}(\mathbb{Z}/2\mathbb{Z})$ has order $6$ and is isomorphic to $S_{3}$, so in this case we obtain an injection $G\rightarrow{S_{3}}$. For the tame case, the semistable covers arising from these extensions were studied in \cite{troptamethree}.
\end{exa}

\begin{exa}
Let $E$ be the elliptic curve given by 
\begin{equation}
y^2=x^3+Ax+B
\end{equation}
for some $A,B\in{K(C)}$ and let $N=3$. The field $K(C)(E[3])$ is then obtained as follows. One first considers the splitting field $L_{f}$ of the polynomial
\begin{equation}
f:=3x^4+6Ax^2+12Bx-A^2.
\end{equation}
This polynomial is also known as the third division polynomial for $E$. It contains the $x$-coordinates for the $3$-torsion. By attaching a square root of $\alpha^{3}+A\alpha+B$ for every root $\alpha$ 
 of $f$, one obtains $K(C)(E[3])$ as a degree $\leq{2}$ extension of $L_{f}$:
\begin{equation}
K(C)\subset{L_{f}}\subset{}K(C)(E[3]).
\end{equation}
The group $\text{SL}_{2}(\mathbb{F}_{3})$ has order $(3^2-3)(3^2-1)/2=24$ and is a solvable group. Let us illustrate this fact.  The group $H$ of matrices of order $4$ (of which there are $6$, which can easily be found using the conditions $\mathrm{tr}(A)=0$ and $\mathrm{det}(A)=1$) is a group of order $8$ and as such it is the unique (normal) Sylow-$2$ group. The quotient group $\text{SL}_{2}(\mathbb{F}_{3})/\{\pm{I}\}$ is then isomorphic to $A_{4}$ (it is the only nonabelian group of order $12$ with a normal Sylow-2 group). The group $A_{4}$ then contains the Klein-4 group as a normal subgroup with cyclic quotient of order three. 
\end{exa}

Let us note that the groups $\mathrm{SL}_{2}(\mathbb{Z}/N\mathbb{Z})$ are usually \emph{not} solvable. For instance, let $N=5$. Then $\text{SL}_{2}(\mathbb{F}_{5})/\{\pm{I}\}$ is isomorphic to $A_{5}$ (see \cite[Theorem 1.2.4]{Bonnaf2011}), which is not solvable. We have the following
\begin{lemma}
Let $S_{p}=\mathrm{PSL}_{2}(\mathbb{F}_{p}):=\mathrm{SL}_{2}(\mathbb{F}_{p})/\{\pm{I}\}$, $p\geq{3}$. Then $S_{p}$ is a simple group if $p\geq{5}$. Every proper subgroup of $S_{p}$ is solvable or isomorphic to the alternating group $A_{5}$: the last possibility occurs only if $p=\pm{1}\bmod{5}$. 
\end{lemma}
\begin{proof}
See \cite[Chapter XX]{Burns1} and \cite[IV-23, Lemma 1]{Ser2}.
\end{proof}

In this paper, we will be mostly interested in \emph{geometric} extensions of $K(C)$. Let $G_{K(C)}$ be the absolute Galois group. By restricting automorphisms of $\overline{K(C)}/K(C)$ to $\overline{K}$, we obtain the natural exact sequence 
\begin{equation}
1\rightarrow{}G_{\overline{K}(C)}\rightarrow{}G_{K(C)}\rightarrow{G_{\overline{K}/K}}\rightarrow{1},
\end{equation}
where $G_{\overline{K}(C)}$ is the absolute Galois group of $\overline{K}(C):=\overline{K}\otimes_{K}{K(C)}$ (see \cite[Chapter 3, Corollary 2.14]{liu2}).
\begin{mydef}\label{GeometricAutomorphisms}
Let $K(C)\subset{L}$ be any finite Galois extension. We say that $\sigma\in{G_{K(C)}}$ (or its image in $G_{K(C)}/G_{L}$) is {\bf{geometric}} on $L$ if $\sigma|_{\overline{K}\cap{L}}=\mathrm{id}$. If the representation $\rho:G\rightarrow{\mathrm{SL}_{2}(\mathbb{Z}/N\mathbb{Z}})$ induced by the $N$-torsion of an elliptic curve consists only of geometric automorphisms, we say that $\rho$ is geometric. 
\end{mydef}

By extending the base field $K$ to $K':=\overline{K}\cap{K(C)(E[N])}$, we can create a canonical geometric representation associated to $K(C)\subset{K(C)(E[N])}$. We will study these extensions in more detail in the next section.     

\subsection{Associating semistable coverings to torsion point extensions}\label{Semistablecovering}

The extension $K(C)\rightarrow{K(C)(E[N])}$ created in Section \ref{Torsionextensions} induces a finite Galois morphism of curves by taking the normalization $C_{N}$ of $C$ in $K(C)(E[N])$. This curve $C_{N}$ is not immediately geometrically irreducible however, as the following example will illustrate. 
\begin{exa}
Let $E$ be any elliptic curve defined over $K$. We then consider $E$ as an elliptic curve over $K(C)$ and consider the extensions $K(C)\subset{K(C)(E[N])}$. These extensions are just finite extensions of the ground field $K$ and as such the intersection of $K(C)(E[N])$ with $\overline{K}$ will be larger than $K$ as soon as the induced extensions $K\subset{K(E[N])}$ are nontrivial. By \cite[Chapter 3, Corollary 2.14]{liu2}, we find that the corresponding normalization of $C$ in $K(C)(E[n])$ is not geometrically irreducible in this case. 
\end{exa}

To obtain a bonafide extension of geometrically irreducible curves, we thus have to extend the ground field $K$. For any fixed $n\in\mathbb{N}$, we take the finite extension $K\subset{K':=K(C)(E[N])\cap{\overline{K}}}$ and consider the base change $C_{K'}:=C\times_{\text{Spec}(K)}{\text{Spec}(K')}$. Note that $C$ is geometrically irreducible, so $C_{K'}$ is again a geometrically irreducible, smooth and projective curve over $K'$. We write $C:=C_{K'}$, $K:=K'$ and take the normalization $C_{N}$ of $C$ in $K(C)(E[N])$. By \cite[Chapter 3, Corollary 2.14]{liu2}, we then have that $C_{N}$ is geometrically irreducible and we thus obtain a finite Galois covering
\begin{equation}
C_{N}\rightarrow{C}
\end{equation} 
with Galois group $G$. Note that this group might have changed after the extension $K\subset{K'}$. 

Let $\mathcal{C}$ be a strongly semistable model for $C$, as defined in Section \ref{DisBran3}. By applying the constructions in that section, we find a finite extension $K'\supset{K}$, a strongly semistable model $\mathcal{C}_{N}$ for $C_{N}$ over $R'$, an induced Galois covering of strongly semistable models
\begin{equation}
\mathcal{C}_{N}\rightarrow{\mathcal{C}}
\end{equation}
over $R'$
and an induced Galois covering of metrized complexes of $k$-curves
\begin{equation}
\Sigma(\mathcal{C}_{N})\rightarrow{\Sigma(\mathcal{C})}.
\end{equation}

In particular, we see that the matrices arising from the \emph{geometric} representation (as in Definition \ref{GeometricAutomorphisms})
\begin{equation}
\rho_{N}:G\rightarrow{\text{SL}_{2}(\mathbb{Z}/N\mathbb{Z})}
\end{equation}
act on the metrized complex $\Sigma(\mathcal{C}_{N})$ such that $\Sigma(\mathcal{C}_{N})/G=\Sigma(\mathcal{C})$. We will be interested in the images of the various decomposition and inertia groups of the edges and vertices of these complexes under this representation. 

\begin{mydef}\label{UnramifiedEdges}
Let $e'$ be an edge of $\Sigma(\mathcal{C}_{N})$ lying above $e\in\Sigma(\mathcal{C})$ for the morphism $\mathcal{C}_{N}\rightarrow{\mathcal{C}}$ constructed above. We say that $E[N]$ is {\emph{unramified}} above an edge $e\in{\Sigma(\mathcal{C})}$ if we have that $I_{e'/e}=(1)$. 
\end{mydef}
\begin{rem}
Note that this definition does not depend on the edge chosen, since any other edge $e''$ above $e$ gives a conjugated inertia group, which again must be trivial. 
\end{rem}

\section{Minimal Weierstrass models and strongly semistable models}

In this section, we will study the notion of minimal Weierstrass models over the function field $K(C)$ of a smooth, geometrically irreducible, projective curve $C$ with strongly semistable model $\mathcal{C}$. We will start with a strongly semistable regular model and then also define the corresponding notion for strongly semistable models. We will be mostly interested in defining the reduction type of an elliptic curve at subgraphs $T$ of $\Sigma(\mathcal{C})$, which we will do in Section \ref{ReductionType2}. We then show some basic properties regarding these reduction types. In Section \ref{NeronCriterion}, we show using Theorem \ref{SwitchTheorem} and the N\'{e}ron-Ogg-Shafarevich criterion (Theorem \ref{NOSCriterion}) that the torsion subgroup $E[N]$ of an elliptic curve with good reduction at an edge is unramified above that edge. 

\subsection{Minimal Weierstrass models}

We first quickly review the notion of a minimal Weierstrass model at a codimension one point of $\mathcal{C}$. Let $E/K(C)$ be an elliptic curve over the function field $K(C)$ of a smooth, geometrically irreducible, projective curve over $K$ with strongly semistable \emph{regular} model $\mathcal{C}$. Let $z\in\mathcal{C}$ be any point of codimension one, i.e. $\text{dim}(\mathcal{O}_{\mathcal{C},z})=1$. We let $v_{z}(\cdot{})$ be the corresponding discrete valuation on $K(C)^{*}$, $R_{z}$ the valuation ring, $\mathfrak{m}_{z}\subset{R_{z}}$ the maximal ideal and $k(z):=R_{z}/\mathfrak{m}_{z}$ the residue field.

Any elliptic curve $E/K(C)$ can be put in the following Weierstrass form
\begin{equation}
y^2+a_{1}xy+a_{3}y=x^3+a_{2}x^2+a_{4}x+a_{6},
\end{equation}
where $a_{i}\in{K(C)}$, see \cite[Chapter III, Proposition 3.1]{silv1}. We will denote the corresponding projective curve by $W_{K(C)}$. By scaling the coefficients, we then obtain a plethora of models over $R_{z}$. We now single out the models that minimize the valuation of the discriminant under the additional condition that $a_{i}\in{R_{z}}$. 
\begin{mydef}
Let $W_{K(C)}$ be a Weierstrass model for $E$ with coefficients $a_{i}\in{R_{z}}$. If the valuation of the discriminant $v_{z}(\Delta_{W})$ is minimal among all Weierstrass models over $K(C)$ with $a_{i}\in{R_{z}}$, then $W_{K(C)}$ is said to be a \emph{minimal Weierstrass model} over $R_{z}$. The valuation $v_{z}(\Delta_{W})$ is referred to as the valuation of the minimal discriminant at $z$.  
\end{mydef} 

\begin{rem}
Let us now relate this theory of minimal Weierstrass models to the scheme-theoretic minimal regular model. Given such a minimal Weierstrass model, we define the scheme-theoretic model $W/R_{z}$ by homogenizing the Weierstrass equation:
\begin{equation}
W:=\text{Proj}(R_{z}[X,Y,Z]/(Y^2Z+a_{1}XYZ+a_{3}YZ^2-X^3-a_{2}X^2Z-a_{4}XZ^2-a_{6}Z^3)).
\end{equation}
The minimal regular model $X/R_{z}$ of $E$ is then related to $W$ as follows. Consider the set of vertical prime divisors $\Gamma$ of $X$ that do not meet the closure of the identity $\{\mathcal{O}\}$ in $X$. There is then a contraction morphism of these components $X\rightarrow{W'}$ and we in fact have that $W=W'$, see \cite[Chapter 9, Theorem 4.35]{liu2}. Here we use that $\text{Pic}(\text{Spec}(R_{z}))=0$. 

The minimal regular model $X$ of $E/K(C)$ over $R_{z}$ is then also related to the N\'{e}ron model $\mathcal{E}/R_{z}$ of $E/K(C)$ by the following procedure: consider the reduced closed subscheme $S$ of $X$ consisting of the nonsmooth points in the closed fiber. The open subscheme $\mathcal{E}:=X\backslash{S}$ then gives the N\'{e}ron model of $E/K(C)$, see \cite[Chapter 10, Theorem 2.14]{liu2}.
\end{rem}

Let $S$ be any collection of valuations of the function field $K(C)$ arising from the model $\mathcal{C}$. That is, for every valuation $v(\cdot{})$ in $S$, there is a point $z$ in $\mathcal{C}$ of codimension one (i.e., $\text{dim}(\mathcal{O}_{\mathcal{C},z})=1$) whose associated valuation gives $v(\cdot{})$. We will write valuations in $S$ as $v_{z}(\cdot{})$, where $z$ is the corresponding point of codimension one in $\mathcal{C}$. For every valuation $v_{z}(\cdot{})$, we can find a minimal Weierstrass model $W_{z}$ with corresponding discriminant $\Delta_{z}$. The valuation $v_{z}(\Delta_{z})$ of the minimal discriminant is then independent of the minimal Weierstrass model chosen.
\begin{mydef}
Let $W/K(C)$ be a Weierstrass model for an elliptic curve $E$ with associated coefficients $\{a_{i}\}$. We say that $W$ is $S$-minimal if for every valuation $v_{z}\in{S}$, we have that $v_{z}(a_{i})\geq{0}$ and $v_{z}(\Delta_{W})=v_{z}(\Delta_{z})$, where $\Delta_{z}$ is the discriminant of any minimal Weierstrass model $W_{z}$ at $z$. 
\end{mydef}

We first note that these global Weierstrass models don't exist if $S$ is too large.
\begin{exa}
Let $C:=\mathbb{P}^{1}_{K}$ with function field $K(w)$ and strongly semistable model $\mathbb{P}^{1}_{R}$.  
Now consider the elliptic curve $E/K(w)$ defined by the Weierstrass equation
\begin{equation}
y^2=x^3+x^2+w.
\end{equation}
The discriminant of this Weierstrass model $W$ is $$\Delta_{W}=-432\cdot{}w^2 - 64\cdot{w}.$$ 

Let $S$ be the set of valuations arising from points in $C$. $W$ is then minimal at every point except infinity. Any transformation (as in \cite[Chapter III, Page 44]{silv1}) that makes $W$ minimal at infinity transforms the discriminant by
\begin{equation}
u^{12}\Delta'=\Delta.
\end{equation}
We then find that the valuation of $u$ at infinity is $-1$, meaning that the valuation of $u$ at some other point must be positive (since the degree of $u$ in $\mathbb{P}^{1}$ must be zero). For any such point, the corresponding Weierstrass model is not minimal, so no such model can exist. 
If instead we take $S$ to be all points minus infinity, then $W$ is $S$-minimal. 
\end{exa}

For the applications we have in mind, $S$ will be finite. 
In this case, global $S$-models always exist.
\begin{pro}\label{SMinimal}
Suppose that $S$ is finite. Then there exists an $S$-minimal Weierstrass model for $E$.
\end{pro}
\begin{proof}
We will follow \cite[Chapter VIII, Proposition 8.2]{silv1} and use the approximation theorem for valuations. Consider the subring $R_{S}\subset{K(C)}$ defined by
\begin{equation}
R_{S}:=\{x\in{K(C)}:v_{z}(x)\geq{0}\text{ for every }v_{z}\in{S}\}.
\end{equation}
That is, it is the intersection of the valuation rings arising from $S$. We first take any Weierstrass model $W$ (not necessarily minimal) defined over $R_{S}$. For every $v_{z}\in{S}$, there then exist $u_{z}, r_{z},s_{z},t_{z}$ such that the transformation defined by
\begin{align}
x&=u_{z}^{2}x_{z}+r_{z}\\
y&=u_{z}^{3}y_{z}+s_{z}u_{z}^{2}x_{z}+t_{z}
\end{align}
gives a minimal model at $z$. We will denote the resulting $a_{i}$ by $a_{i,z}$. Using the approximation theorem for valuations (see \cite[Chapter 9, Lemma 1.9.b]{liu2}), we then first find a $u\in{K(C)}$ such that $v_{z}(u)=v_{z}(u_{z})$ for every $v_{z}\in{S}$. Another application of the approximation theorem then yields 
 $r,s,t\in{R_{S}}$ such that for every $v_{z}\in{S}$, we have 
\begin{equation}
v_{z}(r-r_{z}),v_{z}(s-s_{z}),v_{z}(t-t_{z})>\max_{i=1,2,3,4,6}v_{z}(u_{z}^{i}a_{i,z}).
\end{equation}

A routine check using the transformation formulas for the coefficients $a_{i}$ (see \cite[Chapter III, Page 45]{silv1}) for the transformation defined by 
\begin{align}
x&=u^2{x'}+r\\
y&=u_{}^{3}y'_{}+s_{}u_{}^{2}x'_{}+t_{}
\end{align}
then shows that $v_{z}(a_{i}')\geq{0}$. The discriminant is easily seen to be minimal at every valuation, yielding the desired result. 
\end{proof}

\subsection{$S$-models for intersection graphs}\label{ReductionType2}

In this section, we will study the reduction of an elliptic curve at the vertical divisors of a strongly semistable regular model. We will then be able to study the reduction type of the elliptic curve on subgraphs of the intersection graph $\Sigma(\mathcal{C})$. The notions of good and multiplicative reduction will then be shown to be stable under finite extensions of $K$. 

Let $\mathcal{C}$ be a strongly semistable regular model. Every irreducible component $\Gamma$ in the special fiber of $\mathcal{C}$ gives rise to a valuation which we will denote by $v_{\Gamma}$. Let $S:=\{v_{\Gamma}\}$ be the finite set of all valuations arising from irreducible components $\Gamma$ in the special fiber of $\mathcal{C}$. By Proposition \ref{SMinimal}, we can find a Weierstrass model $W/K(C)$ for $E/K(C)$ that is minimal at every component $\Gamma$. Using this global model, we will now define the reduction type of $E$ on a subgraph of $\Sigma(\mathcal{C})$. We will do this in terms of the normalized Laplacians of the coefficient $c_{4}$ and the discriminant $\Delta$. For more details on the normalized Laplacian, see Definition \ref{NormalizedLaplacian} and Section \ref{NormalizedLaplacianSection}.  For the definition of the inequalities, see Definition 
\ref{InequalitiesDefinition}. 
\newpage
\begin{mydef}
Let $W/K(C)$ be a minimal $S$-model for the intersection graph $\Sigma(\mathcal{C})$ and let $T$ be a subgraph of $\Sigma(\mathcal{C})$.
\begin{enumerate}\label{ReductionType}
\item We say that $E$ has good reduction on $T$ if $\phi_{\Delta}=0$ on $T$.
\item We say that $E$ has multiplicative reduction on $T$ if $\phi_{\Delta}>0$ and $\phi_{c_{4}}=0$ on $T$.
\item We say that $E$ has additive reduction on $T$ if $\phi_{\Delta}>0$ and $\phi_{c_{4}}>0$ on $T$.
\end{enumerate} 
\end{mydef}
\begin{lemma}
The above definition for the reduction type of $E$ on a subgraph $T\subset{\Sigma(\mathcal{C})}$ does not depend on the choice of a minimal $S$-model $W/K(C)$. 
\end{lemma}
\begin{proof}
For any other minimal $S$ model $W'$, we have
\begin{align}
v_{\Gamma}(c'_{4})=v_{\Gamma}(c_{4}),\\
v_{\Gamma}(\Delta')=v_{\Gamma}(\Delta)
\end{align}
for every irreducible component $\Gamma\subset{\mathcal{C}_{s}}$. 
This implies that the normalized Laplacians are the same, so the Lemma follows.
\end{proof}

Let $\mathcal{C}$ be a strongly semistable (not necessarily regular) model for $C$ and let $T$ be any subgraph of $\Sigma(\mathcal{C})$. The minimal desingularization $\mathcal{C}'$ of $\mathcal{C}$ then yields a regular model with a morphism $\mathcal{C}'\rightarrow{\mathcal{C}}$. The graph $T$ then naturally induces a natural subgraph $T'$ of $\Sigma(\mathcal{C}')$ by subdividing every edge of length $k$ into $k$ edges (with vertices in between) of length $1$. We use this to define reduction types on subgraphs of strongly semistable models. 

\begin{mydef}
Let $\mathcal{C}$ be a strongly semistable (not necessarily regular) model for $C$. We say that $E$ has good, multiplicative or additive reduction on a subgraph $T\subset{\Sigma(\mathcal{C})}$ if the desingularization $\mathcal{C}'$ has the corresponding reduction type on the induced subgraph $T'\subset{\Sigma(\mathcal{C})}$.  
\end{mydef}

\begin{lemma}\label{StableDefinition}
The notions of good and multiplicative reduction on subgraphs $T$ of $\Sigma(\mathcal{C})$ are stable under finite extensions $K\subset{K'}$. 
\end{lemma}
\begin{proof}
Let $W/K(C)$ be a minimal model on $\Sigma(\mathcal{C})$, see Proposition \ref{SMinimal}. We claim that $W/K'(C)$ is again minimal. Indeed, the conditions $\phi_{T, a_{i}}\geq{0}$, $\phi_{T, \Delta}=0$ and $\phi_{T, c_{4}}=0$ are stable under finite extensions by Lemma \ref{ScalingLemma} and the definition of the extended normalized Laplacian.  
This then also yields the stability of good and multiplicative reduction, as desired.  
\end{proof}

\begin{lemma}\label{ComponentsReductionType}
Suppose that $E$ has good or multiplicative reduction on a subgraph $T$ of the intersection graph $\Sigma(\mathcal{C})$ a strongly semistable model $\mathcal{C}$. Then for every subdivision $T'$ of $T$ over an extension of $K$, $E$ has the same corresponding reduction type at every vertex $v\in{T'}$.  
\end{lemma}
\begin{proof}
By definition, we have that the value of the normalized Laplacians is equal to the valuation of the function at the component. By \cite[Chapter VII, Proposition 5.1]{silv1}, we see that the conditions on the Laplacian for a vertex correspond exactly with the ordinary notions of good and multiplicative reduction. This immediately yields the Lemma. 
\end{proof}
\begin{lemma}
Suppose that $E$ has good reduction on a subgraph $T$ of $\Sigma(\mathcal{C})$. Then $E$ also has good reduction on the corresponding complete subgraph $\tilde{T}$. 
\end{lemma}
\begin{proof}
Let $W$ be any $S$-minimal model and suppose that $E$ has good reduction on $T$. For any vertex with component $\Gamma$ adjacent to an edge in $T$, we must have $v_{\Gamma}(\Delta)=0$, since the Laplacian is is zero on that edge. This gives the Lemma.  
\end{proof}

In other words, having good reduction is a "closed" condition. Having multiplicative reduction is an open condition however, as the following example will show. 

\begin{exa}
Let $C:=\mathbb{P}^{1}_{K}$ and let $\mathcal{C}:=\text{Proj}(R[Z,T,W]/(ZT-\pi\cdot{W}^{2})$ be a strongly semistable regular model for $C$ with affine open $R[x,t]/(zt-\pi)$. Consider the elliptic curve $E/K(C)$ defined by the Weierstrass equation
\begin{equation}
y^2=x^3+x^2+t.
\end{equation}
The discriminant of this model is then given by
\begin{equation}
\Delta_{W}=-432\cdot{}t^2 - 64\cdot{t}.
\end{equation}
Let $\mathfrak{p}_{1}:=(z)$ and $\mathfrak{p}_{2}=(t)$ be the generic points of the irreducible components $\Gamma_{1}$ and $\Gamma_{2}$ of $\mathcal{C}_{s}$. Then $v_{\mathfrak{p}_{1}}(\Delta_{W})=0$ and $v_{\mathfrak{p}_{2}}(\Delta_{W})=1$. We thus see that $W$ defines an $S$-minimal model and that $E$ has good reduction on $T_{1}:=\{\Gamma_{1}\}$ and multiplicative reduction on $T_{2}:=\{e,\Gamma_{2}\}$.  
\end{exa}

We now study the following problem. Let $E/K(C)$ be an elliptic curve and let $\mathcal{C}'\rightarrow{\mathcal{C}}$ be a disjointly branched morphism. Let $e'\in\Sigma(\mathcal{C}')$ be an edge mapping to $e\in\Sigma(\mathcal{C})$.  
\begin{lemma}\label{DisBranReduction}
Let $\mathcal{C}'\rightarrow{\mathcal{C}}$ be a disjointly branched morphism, as defined in Section \ref{DisBran3}, and suppose that $E/K(C)$ has good (resp. multiplicative) reduction on $e\in\Sigma(\mathcal{C})$. Then $E$ has good (resp. multiplicative) reduction on $e'$ for any $e'\in\Sigma(\mathcal{C}')$ mapping to $e$.
\end{lemma}
\begin{proof}
Let $W/K(C)$ be a minimal model for $\mathcal{C}$. Since $\psi:\mathcal{C}'\rightarrow{\mathcal{C}}$ is disjointly branched, it is vertically unramified. This implies that for any generic point $\mathfrak{p}'\in\mathcal{C}'_{s}$ mapping to $\mathfrak{p}\in\mathcal{C}$, we have that $v_{\mathfrak{p}'}|_{K(C)}=v_{\mathfrak{p}}$. In other words, the vertical valuations are unchanged. This implies that $W$ is minimal over $\mathcal{C}'_{s}$ and that the normalized Laplacians are invariant, which tells us that the reduction type doesn't change. This yields the Lemma.   
\end{proof}

\subsection{N\'{e}ron-Ogg-Shafarevich for edges}\label{NeronCriterion}

In this section, we will use the notion of good reduction developed in the previous section to show that the disjointly branched morphism arising from the $N$-torsion of an elliptic curve with good reduction on an edge of $\Sigma(\mathcal{C})$ is unramified above that edge. We will use Theorem \ref{SwitchTheorem} to prove this.

Let us first give a review of the original N\'{e}ron-Ogg-Shafarevich criterion for codimension one points. Let $\overline{K(C)}$ be the algebraic closure of $K(C)$ and fix a discrete valuation ring $R_{\mathfrak{p}}\subset{K(C)}$ with residue field $k(\mathfrak{p})$. We will consider several $R_{\mathfrak{p}}$-subalgebras in the algebraic closure.  Let $R_{\mathfrak{p}}^{sh}$ be a strict Henselization of $R_{\mathfrak{p}}$ and let $\overline{R}_{\overline{\mathfrak{p}}}$ be the localization of the integral closure of $R$ in $\overline{K(C)}$ at a maximal ideal $\overline{\mathfrak{p}}$ lying over the unique maximal ideal of $R_{\mathfrak{p}}^{sh}$. We let $K(C)^{sh}$ be the quotient field of the strict Henselization. The inertia group $I_{\overline{\mathfrak{p}}/\mathfrak{p}}$ of $\overline{\mathfrak{p}}$ over $\mathfrak{p}$ is then equal to the Galois group $\text{Gal}(\overline{K(C)}/K(C)^{sh})$. That is, $K(C)^{sh}$ is the maximal unramified extension under $\overline{\mathfrak{p}}$.

Let $A$ be an abelian variety over $K(C)$ and let $\mathcal{A}_{\mathfrak{p}}$ be its N\'{e}ron model over $R_{\mathfrak{p}}$, which exists by \cite[Section 1.3, Corollary 2]{Bosch1990}. $A$ is said to have good reduction at $\mathfrak{p}$ if the identity component $\mathcal{A}_{\mathfrak{p},0,s}$ is an abelian variety. The N\'{e}ron-Ogg-Shafarevich criterion can be stated as follows:
\begin{theorem}{\bf{(N\'{e}ron-Ogg-Shafarevich criterion)}}\label{NOSCriterion}
Let $A$ be an abelian variety, $\mathcal{A}_{\mathfrak{p}}$ the N\'{e}ron model for $A$ over $R_{\mathfrak{p}}$ and $\ell$ a prime different from $\text{char}(k(\mathfrak{p}))$.  Then the following are equivalent: 
\begin{enumerate}
\item $A$ has good reduction at $\mathfrak{p}$.
\item $\mathcal{A}_{\mathfrak{p}}$ is an abelian scheme.
\item For every $n\geq{0}$, the inertia group $I_{\overline{\mathfrak{p}}/\mathfrak{p}}$ acts trivially on $A[\ell^{n}](\overline{K(C)})$. In other words, the canonical map
\begin{equation}
A[\ell^{n}](K(C)^{sh})\rightarrow{A[\ell^{n}](\overline{K(C)})}
\end{equation}
is bijective.
\item The Tate module $T_{\ell}(A)$ is unramified over $R$. That is, the inertia group operates trivially on $T_{\ell}(A)$.
\end{enumerate}
\end{theorem} 
\begin{proof}
See \cite[Chapter VII, Theorem 5]{Bosch1990} for the current version and \cite[Chapter VII, Theorem 7.1]{silv1} for the elliptic curve version over a perfect field. 
\end{proof}
\begin{rem}\label{ExtraRemark}
We will be needing the following version, which follows in exactly the same way from \cite[Chapter VII, Theorem 5]{Bosch1990}:
\begin{center}
 Let $N$ be a positive integer such that $\mathrm{char}(k(\mathfrak{p}))\nmid{N}$. Suppose that $A$ has good reduction at $\mathfrak{p}$. Then the canonical map $A[N](K(C)^{sh})\rightarrow{A[N](\overline{K(C)})}$ is bijective.
 \end{center}
\end{rem}
\begin{theorem}\label{MainTheoremGoodReduction} {\bf{(Good reduction for edges)}}
Let $E/K(C)$ be an elliptic curve over a curve $C$ with strongly semistable model $\mathcal{C}$ and intersection graph $\Sigma(\mathcal{C})$. Let $N$ be an integer with $\text{char}(k)\nmid{N}$ and suppose that $E$ has good reduction on an edge $e$ of $\Sigma(\mathcal{C})$. Then $E[N]$ is unramified over $e$,  in the sense of Definition \ref{UnramifiedEdges}. 
\end{theorem}
\begin{proof}
Suppose for a contradiction that $E[N]$ is ramified over $e=vw$, where $v$ corresponds to an irreducible component $\Gamma_{0}$ in $\mathcal{C}_{s}$. By \cite[Proposition 5.1.1]{tropabelian}, we find that the length of any edge $e'$ is smaller than that of $e$. This implies that $\mathcal{C}$ is not regular at $e$. We take a desingularization $\mathcal{C}_{0}$ above $e$ and consider the normalization $\mathcal{C}'_{0}$ of $\mathcal{C}_{0}$ in $K(C)(E[N])$. Let $v_{1}$ be a vertex corresponding to an irreducible component $\Gamma_{1}$ in $\mathcal{C}_{0}$ intersecting $\Gamma_{0}$ and let $\mathfrak{p}_{1}$ be its generic point. Using Lemma \ref{ComponentsReductionType}, we see that $E$ again has good reduction at $\mathfrak{p}_{1}$. By the N\'{e}ron-Ogg-Shafarevich criterion (Theorem \ref{NOSCriterion} and Remark \ref{ExtraRemark}), we then find that $I_{\mathfrak{p}'_{1}/{\mathfrak{p}_{1}}}=(1)$ for $\mathfrak{p}'_{1}$ lying over $\mathfrak{p}_{1}$. Switching from codimension one to codimension two using Theorem \ref{SwitchTheorem}, we then find that $I_{e'/e}=I_{\mathfrak{p}'_{1}/\mathfrak{p}_{1}}=(1)$, a contradiction. This finishes the proof.   
\end{proof}
\begin{cor}
Suppose that $E$ has good reduction at an edge $e\in\Sigma(\mathcal{C})$ and let $|G|$ be the order of the Galois group $G=\text{Gal}(K(C)(E[N])/K(C))$. There are then $|G|$  edges of length $l(e)$ in $\Sigma(\mathcal{C}_{N})$ lying above $e$.  
\end{cor}

\section{Uniformization}\label{Uniformization}

In this section, we review the Tate uniformization for elliptic curves with split multiplicative reduction. Since our residue fields are not necessarily perfect, we cannot use \cite[Chapter V]{Silv2} directly. The version that we need can be found in \cite{Roquette1}.

Let $E$ be an elliptic curve over a local field $K$. Here, a local field is defined to be a field $K$ with a nontrivial real absolute value $a\mapsto{|a|}$ such that $K$ is complete with respect to the metric induced by this absolute value. We will furthermore assume that this absolute value is nonarchimedean. 

For any $q\in{K}$ with $0<|q|<1$, we consider the group $K^{*}/q^{\mathbb{Z}}$ and note that it can be given the structure of a rigid analytic space or a Berkovich space, see \cite[Chapter V, Section 1]{Fresnel2004} and \cite[Chapter 9, Section 2]{Bosch2014}. The elliptic curves $E/K$ that are isomorphic over $K$ to this object $K^{*}/q^{\mathbb{Z}}$ for some $q$ are then given by the following \emph{uniformization theorem}, which was first proved by Tate in \cite{Tate1971}. 

%
\begin{theorem}\label{UniformizationTheorem}
Let $E/K$ be an elliptic curve. Then $E$ is isomorphic to $K^{*}/q^{\mathbb{Z}}$ for some $q\in{K}$ if and only if the following hold:
\begin{enumerate}
\item The $j$-invariant satisfies $|j|>1$.
\item The Hasse invariant $\gamma(E)\in{K^{*}/(K^{*})^2}$ is trivial.
\end{enumerate}
Furthermore, we have that $|q|=\dfrac{1}{|j|}$.
\end{theorem}
\begin{proof}
See \cite[VIIIa]{Roquette1} for the first two claims and \cite[VII]{Roquette1} for the last claim. 
\end{proof}

Recall that for any elliptic curve $E$ with $j$-invariant not equal to $0$ or $1728$ and over a field $K$ of $\text{char}(K)\neq{2,3}$ given by a Weierstrass equation
\begin{equation}
y^2=x^3-27c_{4}x-54c_{6},
\end{equation}
the Hasse invariant of $E$ is given by the image in $K^{*}/(K^{*})^2$ of the element
\begin{equation}
\gamma'(E)=-\dfrac{c_{4}}{c_{6}}.
\end{equation}

We then also see that any elliptic curve with $|j|>1$ becomes isomorphic to $K^{*}/q^{\mathbb{Z}}$ after at most a quadratic extension, which is given by $K\subset{K(\sqrt{\gamma'(E)})}$.

One of the reasons that this uniformization is so valuable, is that the construction commutes with the action of the Galois group in the following sense. Suppose that we have a finite Galois extension of complete fields $K\subset{L}$ and consider the resulting uniformization $\psi_{L}:E(L)\rightarrow{L^{*}/q^{\mathbb{Z}}}$. For any $\sigma\in{\mathrm{Gal}(L/K)}$, we then have  
\begin{equation}
\sigma(\psi(P))=\psi(\sigma(P)).
\end{equation} 

We will see this property in action in the next section. 

\section{Multiplicative reduction and transvections for edges}

In this section, we will study elliptic curves $E/K(C)$ having multiplicative reduction at an edge $e\in\Sigma(\mathcal{C})$ for some strongly semistable model $\mathcal{C}$, as explained in Section \ref{ReductionType}. Using this, we will see that 
an inertia group $I_{e'/e}$ lying above an edge $e$ 
will contain an element $\sigma$ that maps to the matrix $  \begin{pmatrix}
1 & 1 \\
0&1 \end{pmatrix} $ under the Galois representation map $\rho_{\ell}:G\rightarrow{\text{SL}_{2}(\mathbb{F}_{\ell})}$. This will use the uniformization results introduced in the previous section. We will write $\phi_{j}$ for the normalized Laplacian of the $j$-invariant of $E$ and $\delta_{e}(\phi_{j})$ for the absolute value of the slope of $\phi_{j}$ on any edge $e\in\Sigma(\mathcal{C})$. We invite the reader to compare the upcoming theorem to \cite[Chapter V, Corollary 6.2]{Silv2} and \cite[IV-37, Proposition 1]{Ser2}.
Let us set
\begin{equation}
\tau:=\begin{pmatrix}
1 & 1 \\
0&1 \end{pmatrix}\in\mathrm{SL}_{2}(\mathbb{F}_{\ell}). 
\end{equation}
This $\tau$ will be referred to as a transvection, as in \cite[Chapter V]{Silv2} and \cite[Chapter IV]{Ser2}.
\begin{theorem}\label{MainTheorem2}

Let $E$ be an elliptic curve over $K(C)$ with multiplicative reduction at an edge $e\in\Sigma(\mathcal{C})$ and let $\mathcal{C}_{\ell}\rightarrow{\mathcal{C}}$ be the semistable covering associated to the Galois extension $K(C)\subset{K(C)(E[N])}$ with Galois group $G$ and representation $\rho_{\ell}:G\rightarrow{\mathrm{SL}_{2}(\mathbb{F}_{\ell})}$. Let $p=\text{char}(k)$  and let $\ell\geq{3}$ be a prime with $p\nmid{\#\rho_{\ell}(G)}$ and $\ell\nmid{\delta_{e}(\phi_{j})}$. Let $e'$ be an edge of $\Sigma(\mathcal{C}_{\ell})$ lying above $e$. There then exists a $\sigma\in{I_{e'/e}}$ such that
\begin{equation}
\rho_{\ell}(\sigma)\equiv{
  \begin{pmatrix}
1 & 1 \\
0&1 \end{pmatrix} 
}\mod{\ell}.
\end{equation}

\end{theorem}

\begin{proof}
We start by extending the base field $K(C)$. Consider the Galois extension
\begin{equation}
K(C)\rightarrow{K(C)(\sqrt{\gamma'(E)})}=:L.
\end{equation}
That is, we take a representative of the \emph{Hasse invariant} $\gamma(E)\in{K(C)^{*}/(K(C)^{*})^2}$ (as in Section \ref{Uniformization}) and then take the square root of this representative.
Since the composite of two Galois extensions is again Galois, we find that the extension 
\begin{equation}
K(C)\rightarrow{L(E[\ell])}
\end{equation}
is again Galois and we will denote its Galois group by $G'$. We let this group act on the $\ell$-torsion in the natural way and obtain a (possibly noninjective) representation
\begin{equation}
\rho_{\ell}':G'\rightarrow{\mathrm{SL}_{2}(\mathbb{F}_{\ell})}.
\end{equation}
We will obtain a $\sigma\in{G'}$ that maps to the desired matrix in the Theorem. Note that there is a natural surjective map $s:G'\rightarrow{G}$ such that $\rho_{\ell}'=\rho_{\ell}\circ{s}$, which shows that the restriction of $\sigma$ to the Galois extension $K(C)\subset{K(C)(E[\ell])}$ maps to the same matrix under $\rho$. 

We now extend the morphism $K(C)\rightarrow{L(E[\ell])}$ to a disjointly branched morphism $\mathcal{D}\rightarrow{\mathcal{C}}$. The commutative diagram of function fields can be found below, where $H_{1}$ is the Galois group of $L(E[\ell])\supset{}L$ and $H_{2}$ the Galois group of $L(E[\ell])\supset{}K(C)(E[\ell])$.  
\begin{center}
\begin{tikzcd}[every arrow/.append style={dash}]
L \arrow[rightarrow]{r}{H_{1}} & L(E[\ell])\\
K(C) \arrow[rightarrow]{u} \arrow[rightarrow]{r}{G} & K(C)(E[\ell]) \arrow[rightarrow]{u}{H_{2}}
\end{tikzcd}
\end{center}

This then gives rise to a corresponding commutative diagram of semistable models over $R$:

\begin{center}
\begin{tikzcd}[every arrow/.append style={dash}]
\mathcal{D}^{H_{1}} \arrow[rightarrow]{d}& \mathcal{D}\arrow[rightarrow]{l} \arrow[rightarrow]{d}\\
\mathcal{C}  & \arrow[rightarrow]{l}  \mathcal{D}^{H_{2}}
\end{tikzcd}
\end{center}

We write $e_{H_{2}}$ for $e'\in\Sigma(\mathcal{D}^{H_{2}})$ and $e_{(1)}$ for an edge in $\Sigma(\mathcal{D})$ lying above $e_{H_{2}}$. Similarly, we let $e_{H_{1}}$ be the image of $e_{(1)}$ under the natural map $\Sigma(\mathcal{D})\rightarrow{\Sigma(\mathcal{D}^{H_{1}})}$ and $e_{G}=e$, so that we have labeled all relevant edges by subgroups in $G'$. We will now find a $\sigma$ in $I_{e_{H_{1}}/e_{(1)}}$ that maps to the matrix $\tau$ under $\rho'_{\ell}$. By restricting this $\sigma$ to $K(C)(E[\ell])$, we then obtain an element $\tilde{\sigma}\in{I_{e_{H_{2}}/e_{G}}}$ that maps to  $\tau$ 
under $\rho_{\ell}$.


Let $\Gamma_{0}$ be an irreducible component corresponding to a vertex of $e_{H_{1}}$ with generic point $\mathfrak{p}_{0}$. 
 We now extend $\mathcal{D}^{H_{1}}$ to a regular model $\mathcal{D}_{0}^{H_{1}}$ over $e$ and take the normalization $\mathcal{D}_{0}$ of $\mathcal{D}_{0}^{H_{1}}$ in $L(E[\ell])$, as in Theorem \ref{SwitchTheorem}. If $\mathcal{D}^{H_{1}}$ is already regular at $e_{H_{1}}$, we base change to $R(\pi^{1/n})\subset{}K(\pi^{1/n})$ for some $n>1$ (which doesn't affect disjointly branched morphisms, in the sense that vertical components stay unramified) and then take a desingularization $\mathcal{D}_{0}^{H_{1}}$ above $e_{H_{1}}$.\footnote{We will in fact find that $\mathcal{D}^{H_{1}}$ cannot be regular at $e_{H_{1}}$, since the length of $e_{(1)}$ is multiplied by $\ell$ under the map $\mathcal{D}\rightarrow{\mathcal{D}^{H_{1}}}$.}  Let $\mathfrak{p}_{1}\in{\mathcal{D}_{0,s}^{H_{1}}}$ be the generic point of a component intersecting $\Gamma_{0}$ and let $\mathfrak{q}_{1}$ be any point of $\mathcal{D}_{0}$ lying above it. By Theorem \ref{SwitchTheorem}, we find that $I_{\mathfrak{q}_{1}/\mathfrak{p}_{1}}=I_{e_{(1)}/e_{H_{1}}}$. We will find a $\sigma\in{I_{\mathfrak{q}_{1}/\mathfrak{p}_{1}}}$ that maps to the transvection $\tau$. This then also yields the desired $\sigma\in{I_{e_{(1)}/e_{H_{1}}}}$ by the equality $I_{\mathfrak{q}_{1}/\mathfrak{p}_{1}}=I_{e_{(1)}/e_{H_{1}}}$. 

Let us first note that by Lemma \ref{DisBranReduction}, $E$ has multiplicative reduction over $e_{H_{1}}$. Lemma \ref{ComponentsReductionType} then implies that $E$ has multiplicative reduction at $\mathfrak{p}_{1}$.  Writing $L'=L(E[\ell])$, we obtain an inclusion of completed fields
\begin{equation}
L_{\mathfrak{p}_{1}}\rightarrow{L'_{\mathfrak{q}_{1}}}.
\end{equation}
This field extension is again Galois, with Galois group $D_{\mathfrak{q}_{1}/\mathfrak{p}_{1}}$, see \cite[Chapter II, Proposition 9.9]{neu}. Furthermore, the inertia group of this Galois extension is equal to the inertia group $I_{\mathfrak{q}_{1}/\mathfrak{p}_{1}}$. Note that $E$ has split multiplicative reduction at $\mathfrak{p}_{1}$, so for any finite (and thus complete with respect to the unique induced valuation by \cite[Lemma 3.1]{Ste2}) extension $N$ of $L_{\mathfrak{p}_{1}}$ we can write
\begin{equation}
E(N)\simeq{N^{*}}/q^{\mathbb{Z}}
\end{equation}
for some $q\in{L_{\mathfrak{p}_{1}}}$ by Theorem \ref{UniformizationTheorem}. Consider the finite (Galois) extension defined by $L_{\mathfrak{p}_{1}}\subset{L_{\mathfrak{p}_{1}}(q^{1/\ell})}$. Using the adic uniformization, one then easily finds that $L_{\mathfrak{p}_{1}}(q^{1/\ell})\subset{L'_{\mathfrak{q}_{1}}}$. Here one uses that the image of $q^{1/\ell}$ defines an $\ell$-torsion point in $(L_{\mathfrak{p}_{1}}(q^{1/\ell}))^{*}/q^{\mathbb{Z}}$ and that $L(E[\ell])$ injects into $L'_{\mathfrak{q}_{1}}$.

\begin{lemma}
The extension $L_{\mathfrak{p}_{1}}\subset{L_{\mathfrak{p}_{1}}(q^{1/\ell})}$ is totally ramified of degree $\ell$. 
\end{lemma} 
\begin{proof}
$L_{\mathfrak{p}_{1}}\subset{L_{\mathfrak{p}_{1}}(q^{1/\ell})}$ is a Kummer extension, so it suffices to show that $v_{\mathfrak{p}_{1}}(q)$ is not divisible by $\ell$. Suppose that $v_{\mathfrak{p}_{1}}(q)$ is divisible by $\ell$. We will use our assumption $\ell\nmid{\delta_{e}(\phi_{j})}$ on the absolute value of the slope of the Laplacian of the $j$-invariant to obtain a contradiction.  
We have 
\begin{equation}
v_{\mathfrak{p}_{1}}(q)=-v_{\mathfrak{p}_{1}}(j)=-v_{\mathfrak{p}_{0}}(j)\pm\delta_{e}(\phi_{j})
\end{equation}
by Theorems \ref{UniformizationTheorem} and \ref{MainThmVert} respectively. We claim that $v_{\mathfrak{p}_{0}}(j)$ is divisible by $\ell$. Indeed, we either have that $v_{\mathfrak{p}_{0}}(j)=0$ or $v_{\mathfrak{p}_{0}}(j)<0$ by the fact that $E$ has multiplicative reduction on $e$. In the first case, the claim follows. In the second case, note that $\mathcal{D}\rightarrow{\mathcal{D}^{H_{1}}}$ is disjointly branched, so it is unramified above $\mathfrak{p}_{0}$. Using the adic uniformization for $\mathfrak{p}_{0}$, this then easily yields that $v_{\mathfrak{p}_{0}}(q)$ is divisible by $\ell$ (otherwise the extension above $\mathfrak{p}_{0}$ would be ramified). We thus see that $v_{\mathfrak{p}_{1}}(q)\equiv{-\delta_{e}(\phi_{j})}\bmod{\ell}\neq{0}\bmod{\ell}$, a contradiction.  
We thus conclude that $L_{\mathfrak{p}_{1}}\subset{L_{\mathfrak{p}_{1}}(q^{1/\ell})}$ is totally ramified of degree $\ell$.  
\end{proof}

Consider the points $z_{1}=\zeta_{\ell}$ (a primitive $\ell$-th root of unity) and $z_{2}=q^{1/\ell}$ in $(L_{\mathfrak{p}_{1}}(q^{1/\ell}))^{*}/q^{\mathbb{Z}}$. Their inverse images $P_{i}:=\psi^{-1}(z_{i})$ under the uniformization map then define a basis for the $\ell$-torsion of $E$. Using this, one finds that $L'_{\mathfrak{q}_{1}}=L_{\mathfrak{p}_{1}}(q^{1/\ell})$. 

Note that the extension $L_{\mathfrak{p}_{1}}\subset{L_{\mathfrak{p}_{1}}(q^{1/\ell})}$ has a cyclic Galois group, with generating automorphism $\sigma$ given by
\begin{equation}
\sigma(q^{1/\ell})=\zeta_{\ell}\cdot{}q^{1/\ell}.
\end{equation}
We then have
\begin{align}
\sigma(P_{1})&=\sigma(\psi^{-1}(z_{1}))=\psi^{-1}(\sigma(z_{1}))=\psi^{-1}(z_{1})=P_{1}\\
\sigma(P_{2})&=\sigma(\psi^{-1}(z_{2}))=\psi^{-1}(\sigma(z_{2}))=\psi^{-1}(z_{1}\cdot{}z_{2})=P_{1}+P_{2}.
\end{align}
We thus see that the corresponding matrix for $\sigma$ in terms of the basis $\{P_{1},P_{2}\}$ is given by 
\begin{equation}\tau:=\begin{pmatrix}
1 & 1 \\
0&1 \end{pmatrix}.
\end{equation} Note that $\sigma\in{I_{\mathfrak{q}_{1}/\mathfrak{p}_{1}}}$, which gives $\sigma\in{I_{e_{(1)}/e_{H_{1}}}}$ and thus $\sigma\bmod{H_{2}}=\sigma|_{K(C)(E[\ell])}\in{I_{e'/e}}$ by our earlier considerations. This finishes the proof.  
\end{proof}
\begin{cor}
Suppose that the conditions of Theorem \ref{MainTheorem2} are satisfied. There are then $|G|/\ell$ edges lying above $e$, with length $l(e')=l(e)/\ell$. 
\end{cor}

As an application of Theorem \ref{MainTheorem2}, we now show that these geometric transvections of edges generate $\mathrm{SL}_{2}(\mathbb{F}_{\ell})$, provided that $G$ acts irreducibly on $E[\ell]$. By acting irreducibly, we mean that there is no $G$-invariant subspace of the $\mathbb{F}_{\ell}$ vector space $E[\ell]$. Showing that the action of $G$ on $E[\ell]$ is irreducible amounts to showing that there are no $K(C)$-rational isogenies $E\rightarrow{E'}$ of degree $\ell$. 

\begin{theorem}\label{SurjectivityRepresentation}
Let $E$, $G$ and $\ell$ be as in Theorem \ref{MainTheorem2}. Suppose that $G$ acts irreducibly on $E[\ell]$. Then $\rho_{\ell}:G\rightarrow{}\mathrm{SL}_{2}(\mathbb{F}_{\ell})$ is surjective.
\end{theorem}
\begin{proof}
We will write out the proof in \cite[IV-20, Lemma 2]{Ser2} in more detail. For any transvection $\sigma\in{G}$, let $V_{\sigma}$ be the unique one-dimensional subspace of $E[\ell]$ that is fixed by $\sigma$.  Let $P_{1}$ and $P_{2}$ be a basis for $E[\ell]$ such that $\sigma$ acts as $ \begin{pmatrix}
1 & 1 \\
0&1 \end{pmatrix} $ on $\{P_{1},P_{2}\}$. In other words, $\sigma(P_{1})=P_{1}$ and $\sigma(P_{2})=P_{1}+P_{2}$. For any $\tau\in{G}$, we now easily see that $\tau\sigma\tau^{-1}$ acts as a transvection on the adjusted basis $\{\tau(P_{1}),\tau(P_{2})\}$. 
We claim that there are transvections $\sigma$ and $\sigma'$ such that $V_{\sigma}\neq{V_{\sigma'}}$.
Indeed, suppose that $V_{\sigma}=V_{\sigma'}$ for all transvections. Then $V_{\sigma}$ is invariant under $G$. Indeed, for any $\tau\in{G}$, we have that $\tau(V_{\sigma})$ is the invariant subspace of the transvection $\tau\sigma\tau^{-1}$, thus yielding $\tau(V_{\sigma})=V_{\sigma}$. But this contradicts the irreducibility of $G$ on $E[\ell]$, as desired.

We can thus find $\sigma$ and $\sigma'$ such that $V_{\sigma}\neq{V_{\sigma'}}$. For a suitable basis, we then find that 
\begin{equation}
\rho(\sigma)= \begin{pmatrix}
1 & 1 \\
0&1 \end{pmatrix},\,
\rho(\sigma')= \begin{pmatrix}
1 & 0 \\
1&1 \end{pmatrix}.
\end{equation} 
These two matrices generate $\mathrm{SL}_{2}(\mathbb{F}_{\ell})$, which concludes the proof of the theorem.  
\end{proof}

\bibliographystyle{alpha}
\bibliography{bibfiles}{}

\end{document}